\documentclass[a4paper, 12pt, reqno]{amsart}
\usepackage{amsmath, amsthm, amssymb, amsfonts, fullpage,hyperref}
\usepackage[maxbibnames=99,backend=biber,bibencoding=utf8,style=numeric,giveninits=true,url=false,doi=false,isbn=false]{biblatex}
\usepackage{mathtools}
\mathtoolsset{showonlyrefs}
\usepackage{graphicx, subfigure, color,calc}
\numberwithin{equation}{section}
\usepackage{tikz}
\usepackage{pgfplots}
\pgfplotsset{compat=newest}
\usepgfplotslibrary{polar}

\title[Transition Analysis: From the Airy and Pearcey Kernels to the Sine Kernel]{Transition Analysis: From the Airy and Pearcey Kernels to the Sine Kernel}

\author[]{Thorsten Neuschel$^*$}
\address{$^*$School of Mathematical Sciences, Dublin City University}
\email{thorsten.neuschel@dcu.ie}

\author[]{Martin Venker$^\dag$}
\address{$^*$School of Mathematical Sciences, Dublin City University}
\email{martin.venker@dcu.ie}

\newcommand{\R}{\mathbb{R}}

\newcommand{\Ai}{{\rm Ai}}

\newtheorem{thm}{Theorem}[section]

\theoremstyle{remark}
\newtheorem{remark}[thm]{Remark}

\newcommand{\lb}{\left(}
\newcommand{\rb}{\right)}

\renewcommand{\t}{\tau}

\newcommand{\e}{\epsilon}
\newcommand{\w}{\omega}
\newcommand{\z}{\zeta}

\begin{document}

\begin{abstract} 
	We study transitions between the three universal limiting kernels Airy, Pearcey and sine kernel, arising in Random Matrix Theory at edge, cusp and bulk points of the spectrum. Under appropriate rescalings, we provide complete asymptotic expansions of the extended Airy kernel and the extended Pearcey kernel approaching the extended sine kernel, expliciting the fluctuations.
\end{abstract}

\maketitle

\section{Introduction and Statement of Results}

The sine kernel, the Airy  kernel and the Pearcey kernel are central to the theory of universality in random matrix theory and beyond. They describe the limiting correlations of eigenvalues of many classes of random complex Hermitian matrices in the bulk of the spectrum (sine kernel), at the edge of the spectrum (Airy kernel), and at cusps in the spectrum (Pearcey kernel), respectively. Moving from an edge or cusp point into the bulk of the spectrum motivates the following question: is it possible to pass from one kernel to another and can we describe the transitions explicitly in terms of asymptotic expansions?
In this work we answer these questions affirmatively, discussing transitions from the Airy to the sine kernel, and from the Pearcey kernel to the sine kernel, which also allows us to study the fluctuations of both kernels around the sine kernel in detail.

We will consider the kernels in their time-dependent versions, which are usually called \textit{extended} kernels. These arise naturally in the best-understood dynamic model of random matrices, called Non-Intersecting Brownian Motions. To define this process, let $(M(t))_{t\geq0}$ be an $n\times n$ Hermitian Brownian motion, i.e.~$M(t)$ is an $n\times n$ matrix whose diagonal entries are real Brownian motions and off-diagonal entries are complex Brownian motions not necessarily starting in 0, and all Brownian motions on or above the diagonal are independent. Then the eigenvalues process of $(n^{-1/2}M(t))_{t\geq0}$ is called Non-Intersecting Brownian Motions (NIBM), the name stemming from the surprising fact that the eigenvalues of a Hermitian Brownian motion have the same law as $n$ real independent Brownian motions conditioned on non-intersection \cite{Grabiner}. Another name for this process is Dyson's Brownian motion \cite{Dyson}. NIBM is a space-time determinantal process, meaning that the correlations, i.e.~joint intensities, of the point process of eigenvalues at $k$ times and locations can be expressed as determinants of a matrix, each entry of the matrix being an evaluation of a kernel $K_{n,s,t}(x,y)$, where $s,t>0$ are time variables and $x, y\in\R$ space variables. We refer to \cite{CNV2} for details on space-time correlation functions and an explicit form of the kernel $K_n:=K_{n,s,t}$. Here, it suffices to remark that the determinantality of NIBM allows to study the large $n$ asymptotics of the process of the $n$ eigenvalues by studying the asymptotics of $K_n$, which just has four variables apart from $n$. The main objects of this paper arise in a so-called local, or microscopic, rescaling of $s,t,x,y$ according to the type of region of the spectrum one is looking at. Generally speaking, a point in the spectrum is chosen and then spatial variables are rescaled to match the average spacing between consecutive eigenvalues around that point. At an edge point $x_0$ at time $t_0$ this rescaling is of the form 
\begin{align}
	x=x_0+\frac{u}{cn^{2/3}},\quad y=x_0+\frac{v}{cn^{2/3}},\quad u,v\in\R
\end{align} for some constant $c$. Time is then rescaled around $t_0$ as
\begin{align}
	s=t_0+\frac{\t_1}{c'n^{1/3}},\quad t=t_0+\frac{\t_2}{c'n^{1/3}}, \quad \t_1,\t_2\in\R
\end{align} 
for some constant $c'$. With these rescalings the kernel $K_n$ converges as $n\to\infty$ to the \textit{extended Airy kernel} (see \cite{NV} and references therein), defined as 
	\begin{align}\label{extended_Airy}
	\mathbb K^{\rm Ai}_{\t_1,\t_2}(u,v)&:=\frac{1}{(2\pi i)^2} \int_{\Sigma^{\rm Ai}} d\zeta \int_{\Gamma^{\rm Ai}} d\w ~ \frac{\exp\lb\frac{\zeta^3}3-\zeta v-\t_2\zeta^2-\frac{\w^3}3+\w u+\t_1\w^2\rb}{\zeta-\w}\\
	&-1(\t_1>\t_2)\frac{1}{\sqrt{4\pi(\t_1-\t_2)}}\exp\lb-\frac{(u-v)^2}{4(\t_1-\t_2)}\rb,
	\end{align}
	where $\t_1,\t_2\in\R$ are time and $u,v\in\R$ are spatial arguments. The contour $\Sigma^{\Ai}$ consists of the two rays from $\infty e^{-i\frac{\pi}3}$ to 0 and from 0 to  $\infty e^{i\frac{\pi}3}$ and $\Gamma^{\rm Ai}$ consists of the two rays from $\infty e^{-i\frac{2\pi}3}$ to 0 and from 0 to  $\infty e^{i\frac{2\pi}3}$, see Figure \ref{Airy_contours}. For $\t_1,\t_2=0$, the extended Airy kernel $\mathbb K^{\rm Ai}_{0,0}(u,v)$ coincides with the classical Airy kernel
	\begin{align}
		\mathbb K^\Ai(u,v):=\frac{\Ai(u)\Ai'(v)-\Ai'(u)\Ai(v)}{u-v},
	\end{align}
 where $\Ai$ is the Airy function.

	It is worth noting that the definition of the extended Airy kernel is not uniform across the literature and is usually informed by the specific model under consideration. For example, the definitions in \cite{PS,Johansson03,BorodinKuan,Petrov} on random growth models differ from ours by a conjugation and a shift of variables. Our definition is similar to the one in \cite{DuseMetcalfe}. We may speak of variants of the same kernel and remark that up to trivial rescaling they all generate the same stochastic process of infinitely many lines, the Airy line ensemble \cite{CorwinHammond}.
Statistics of the Airy line ensemble have been found for large classes of random matrices and related models, see e.g.~\cite{Soshnikov,PasturShcherbina03,BEY,EYbook,KSSV,KV}. It also appears in a number of studies of interacting particle systems belonging to the KPZ universality class, see e.g.~\cite{Corwinsurvey,SpohnLN,KK} and references therein.

\begin{figure}[h]\label{Airy_contours}
\centering
\begin{tikzpicture}
    \begin{axis}[
        axis lines=middle,
        xlabel={$\textcolor{red}{\Re \zeta}$, $\textcolor{blue}{\Re \omega}$ },
        ylabel={$\textcolor{red}{\Im \zeta}$, $\textcolor{blue}{\Im \omega}$  },
        xmin=-4, xmax=4,
        ymin=-4, ymax=4,
        domain=-4:4,
        samples=100,
        axis equal,
        xtick={-4,-2,0,2,4},
        ytick={-4,-2,0,2,4}
    ]
    \addplot [thick, red, domain=0:4, samples=100] ({x*cos(60)}, {x*sin(60)});
    \addplot [thick, red, domain=0:4, samples=100] ({x*cos(-60)}, {x*sin(-60)});
    
    \addplot [thick, blue, domain=0:4, samples=100] ({x*cos(120)}, {x*sin(120)});
    \addplot [thick, blue, domain=0:4, samples=100] ({x*cos(-120)}, {x*sin(-120)});
    
    
    \end{axis}
\end{tikzpicture}

\caption{Contours of the Airy kernel the complex plane. The red rays \(\Sigma^{\rm Ai}\) correspond to the rays from $\infty e^{-i\frac{\pi}{3}}$ to 0 and from 0 to $\infty e^{i\frac{\pi}{3}}$. The blue rays \(\Gamma^{\rm Ai}\) correspond to the rays from $\infty e^{-i\frac{2\pi}{3}}$ to 0 and from 0 to $\infty e^{i\frac{2\pi}{3}}$.}
\label{fig:contours}
\end{figure}
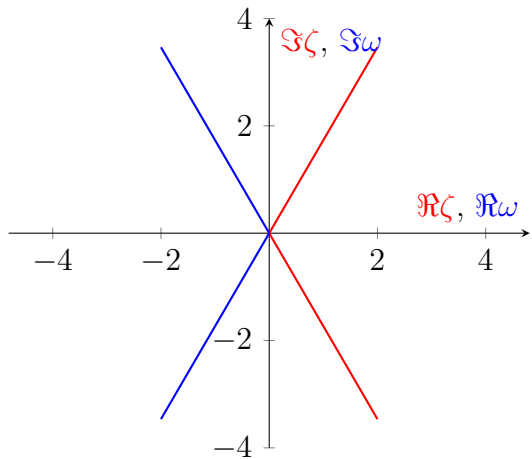

\medskip 

At a cusp point in the spectrum, usually the \textit{extended Pearcey kernel} makes an appearance. 
It is defined as
\begin{align}
	\mathbb K^{\rm P}_{\t_1,\t_2}(u,v):&=\frac{1}{(2\pi i)^2} \int_{i\R} d\zeta \int_{\Gamma^{\rm P}} d\w ~ \frac{\exp\lb-\frac{\zeta^4}4-\frac{\t_2\zeta^2}2-v\zeta +\frac{\w^4}4+\frac{\t_1\w^2}2+u\w\rb}{\zeta-\w}\\
	&-1(\t_1>\t_2)\frac{1}{\sqrt{2\pi(\t_1-\t_2)}}\exp\lb-\frac{(u-v)^2}{2(\t_1-\t_2)}\rb,\label{eq:Pearceykernel}
	\end{align}
	where $\t_1,\t_2\in\R$ are time and $u,v\in\R$ spatial arguments, and $\Gamma^{\rm P}$ consists of four rays, two from the origin to $\pm\infty e^{-i\pi/4}$ and two from $\pm\infty e^{i\pi/4}$ to the origin, see Figure \ref{Pearcey_contours}.  The Pearcey process defined by this kernel has so far been mainly associated with the situation of two bulks of random particles merging, see e.g.~\cite{TracyWidom2,Erdosetal,CapitainePeche,BrezinHikami,BK1,LiechtyWang1,LiechtyWang2,GeudensZhang,OkounkovReshetikin,Erdosetalreal}. The typical rescaling around a cusp $x_0$ at time $t_0$ is spatially 
	\begin{align}
		x=x_0+\frac{u}{cn^{3/4}},\quad y=x_0+\frac{v}{cn^{3/4}},\quad u,v\in\R\label{scaling:Pearcey}
	\end{align} for some constant $c$, and temporally
	\begin{align}
		s=t_0+\frac{\t_1}{c'n^{1/2}},\quad t=t_0+\frac{\t_2}{c'n^{1/2}}, \quad \t_1,\t_2\in\R
	\end{align} 
for some $c'$.

\begin{figure}[h]
\centering
\begin{tikzpicture}
    \begin{axis}[
        axis lines=middle,
        xlabel={$\Re \omega$},
        ylabel={$\Im \omega$},
        xmin=-4, xmax=4,
        ymin=-4, ymax=4,
        domain=-4:4,
        samples=100,
        axis equal,
        xtick={-4,-2,0,2,4},
        ytick={-4,-2,0,2,4}
    ]
    \addplot [thick, blue, domain=0:4, samples=100] ({x*cos(45)}, {x*sin(45)});
    \addplot [thick, blue, domain=0:4, samples=100] ({x*cos(-45)}, {x*sin(-45)});
    
    \addplot [thick, blue, domain=0:4, samples=100] ({-x*cos(45)}, {-x*sin(45)});
    \addplot [thick, blue, domain=0:4, samples=100] ({-x*cos(-45)}, {-x*sin(-45)});
    
    
    \end{axis}
\end{tikzpicture}
\caption{The contour \(\Gamma^P\) of the Pearcey kernel.}
\label{Pearcey_contours}
\end{figure}
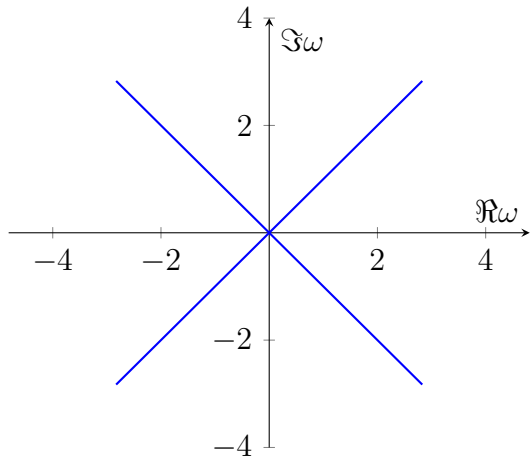

This paper is concerned with convergence of the extended Airy and Pearcey kernels to the limiting bulk kernel, the \textit{extended sine kernel} \cite{TracyWidom04,AvM05}. It is defined as
\begin{align}\label{extended_sine}
	\mathbb K^{\rm sine}_{\t_1,\t_2}(u,v)&:=\frac{1}{2\pi } \int_{-\pi}^{\pi}  ~ e^{\frac{1}{2}(\t_1 - \t_2)\w^2 +i(u-v)\w} d\w\\
	&-1(\t_1>\t_2)\frac{1}{\sqrt{2\pi(\t_1-\t_2)}}\exp\lb-\frac{(u-v)^2}{2(\t_1-\t_2)}\rb,
	\end{align}
where $\t_1,\t_2\in\R$ are time and $u,v\in\R$ spatial arguments.  For \(\t_1 = \t_2 = \t\) the extended sine kernel coincides with the classical sine kernel 
\[\mathbb K^{\rm sine}_{\t,\t}(u,v) = \frac{\sin \pi (u-v)}{\pi (u-v)}.\]

 Like the other kernels, there are different variants of this definition in the literature (the definition in this paper is for example from \cite{KatoriTanemura}). Embracing this variety, for the sake of simpler constants and clearer representation in our theorems and proofs, we introduce the following variant of the sine kernel by
\begin{align}\label{S_1-kernel}
	\mathbb K^{\rm S_1}_{\t_1,\t_2}(u,v)&:=\frac{1}{2\pi i} \int_{-i}^{i}  ~ e^{(\t_1 - \t_2)\w^2 +(u-v)\w} d\w\\
	&-1(\t_1>\t_2)\frac{1}{\sqrt{4\pi(\t_1-\t_2)}}\exp\lb-\frac{(u-v)^2}{4(\t_1-\t_2)}\rb,
	\end{align}
with $u,v,\t_1,\t_2\in\R$. It is easily checked that the two kernels are related by
\begin{equation}\label{connS} \pi \mathbb K^{\rm S_1}_{\frac{\pi^2 \t_1}{2}, \frac{\pi^2\t_2}{2}}(\pi u, \pi v) = \mathbb K^{\rm sine}_{\t_1,\t_2}(u,v), \quad u,v,\t_1,\t_2 \in \mathbb{R}.
\end{equation}
The extended sine kernel usually arises in NIBM around a bulk point $x_0$ at time $t_0$ under the rescaling 
\begin{align}
	x=x_0+\frac{u}{cn},\quad y=x_0+\frac{v}{cn},\quad u,v\in\R
\end{align} for some $c$, with time rescaling
\begin{align}
	s=t_0+\frac{\t_1}{c'n},\quad t=t_0+\frac{\t_2}{c'n}, \quad \t_1,\t_2\in\R
\end{align} 
for some $c'$.

In our first result we study the transition from the extended Airy kernel to the extended sine kernel.

\begin{thm}\label{thm:Airytosine} 
	We have for any $u,v,\t_1, \t_2 \in\R$
	\begin{align}\label{trans:AirytoS}	\frac{1}{\sqrt{a}} & \mathbb  K^{\rm Ai}_{\t_1/a,\t_2/a} \left(\frac{u}{\sqrt{a}} -a, \frac{v}{\sqrt{a}} -a \right)=\mathbb K^{\rm S_1}_{\t_1,\t_2}(u,v) + \mathrm{fluc^{\mathrm S_1}}(u,v; \t_1, \t_2; a) +\mathcal{O}\left(a^{-3}\right)
\end{align}
as \(a\to+\infty\), uniformly in $u,v,\t_1, \t_2 $ when restricted to compact subsets of \(\mathbb{R}\),  where the fluctuations are given by
\[\mathrm{fluc^{\mathrm{S}_1}}(u,v; \t_1, \t_2; a) = \frac{\pi}{(2\pi i)^2} \frac{1}{a^{3/2}} e^{-(\tau_1 -\tau_2)} \left( f_{\tau_1, \tau_2}^{\mathrm{S}_1} (u,v) + \cos\left(\frac{4}{3} a^{3/2} -(u+v)\right)\right),\]
with 
\[f_{\tau_1, \tau_2}^{\mathrm{S}_1} (u,v) =  (u+v)\cos(u-v) -2(\tau_1 + \tau_2) \sin(u-v).\]
More generally, for the Airy to  sine kernel transition there exists a complete asymptotic expansion of the form
\begin{align}\label{Exp:AirytoS} \frac{1}{\sqrt{a}} & \mathbb  K^{\rm Ai}_{\t_1/a,\t_2/a} \left(\frac{u}{\sqrt{a}} -a, \frac{v}{\sqrt{a}} -a \right) \approx \mathbb K^{\rm S_1}_{\t_1,\t_2}(u,v) \\
&+ \sum_{\nu=1}^{\infty} \frac{2}{(2\pi i)^2 a^{\frac{3}{2} \nu}} \sum_{\substack{k,l \geq 0\\ k+l = 2\nu -1}} \Re\left\{b_{k,l} B_{k,l}\right\}\\
&+\sum_{\nu=1}^{\infty} \frac{2}{(2\pi i)^2 a^{\frac{3}{2} \nu}} \sum_{\substack{k,l \geq 0\\ k+l = \nu -1}} \Re\left\{c_{2k,2l}e^{i\frac{4}{3} a^{\frac{3}{2}}}\right\}\Gamma\left(k+\frac{1}{2}\right)\Gamma\left(l +\frac{1}{2}\right), 
\end{align}
as \(a \to +\infty\), uniformly in \(u,v,\t_1, \t_2\), each coming from a compact subset of \(\mathbb{R}\). The coefficients \(b_{k,l}\) are defined in \eqref{varphi_exp}, the coefficients \(B_{k,l}\) are defined in \eqref{def:B}, and the coefficients \(c_{k,l}\) are defined in \eqref{varPhi_exp}.

\end{thm}
\begin{remark} We note that in the context of asymptotic expansions the symbol \(\approx\) in \eqref{Exp:AirytoS} means, that for every integer \(N\geq 0\), we have
\begin{align}\frac{1}{\sqrt{a}} & \mathbb  K^{\rm Ai}_{\t_1/a,\t_2/a} \left(\frac{u}{\sqrt{a}} -a, \frac{v}{\sqrt{a}} -a \right) - \mathbb K^{\rm S_1}_{\t_1,\t_2}(u,v) \\
&- \sum_{\nu=1}^{N} \frac{2}{(2\pi i)^2 a^{\frac{3}{2} \nu}} \sum_{\substack{k,l \geq 0\\ k+l = 2\nu -1}} \Re\left\{b_{k,l} B_{k,l}\right\}\\
&-\sum_{\nu=1}^{N} \frac{2}{(2\pi i)^2 a^{\frac{3}{2} \nu}} \sum_{\substack{k,l \geq 0\\ k+l = \nu -1}} \Re\left\{c_{2k,2l}e^{i\frac{4}{3} a^{\frac{3}{2}}}\right\}\Gamma\left(k+\frac{1}{2}\right)\Gamma\left(l +\frac{1}{2}\right)\\
& = \mathcal{O}\left(a^{-\frac{3}{2}(N+1) }\right), 
\end{align}
as \(a \to +\infty\), uniformly in \(u,v,\t_1, \t_2\), each coming from a compact subset of \(\mathbb{R}\). This constitutes a \textit{complete} (or \textit{full}) asymptotic expansion as it holds up to any order $N$.
\end{remark}

\begin{remark} \label{remark1} According to relation \eqref{connS}, we can restate Theorem \ref{thm:Airytosine} in the following form directly in terms of the sine kernel \eqref{extended_sine}: for $u,v,\t_1, \t_2 \in\R$, we have
\begin{align}\label{trans:Airytosine}	\frac{\pi}{\sqrt{a}} & \mathbb  K^{\rm Ai}_{\frac{\pi^2 \t_1}{2a},\frac{\pi^2 \t_2}{2a}} \left(\frac{\pi u}{\sqrt{a}} -a, \frac{\pi v}{\sqrt{a}} -a \right)=\mathbb K^{\rm sine}_{\t_1,\t_2}(u,v) + \mathrm{fluc}(u,v; \t_1, \t_2; a) +\mathcal{O}\left(a^{-3}\right),
\end{align}
as \(a\to+\infty\), uniformly in $u,v,\t_1, \t_2 $ when restricted to compact subsets of \(\mathbb{R}\),  where the fluctuations are given by
\[ \mathrm{fluc}(u,v; \t_1, \t_2; a)= -\frac{1}{4} \frac{1}{a^{3/2}} e^{-\frac{\pi^2}{2}(\tau_1 -\tau_2)} \left( f_{\tau_1, \tau_2} (u,v) + \cos\left(\frac{4}{3} a^{3/2} -\pi(u+v)\right)\right),\]
with 
\[f_{\tau_1, \tau_2} (u,v) =  \pi (u+v)\cos\left(\pi(u-v)\right) -\pi^2(\tau_1 + \tau_2) \sin\left(\pi(u-v)\right).\]
Moreover, for the Airy to sine kernel transition there exists a complete asymptotic expansion of the form
\begin{align}\label{Exp:AirytoSine}\frac{\pi}{\sqrt{a}} & \mathbb  K^{\rm Ai}_{\frac{\pi^2 \t_1}{2a},\frac{\pi^2 \t_2}{2a}} \left(\frac{\pi u}{\sqrt{a}} -a, \frac{\pi v}{\sqrt{a}} -a \right) \approx \mathbb K^{\rm sine}_{\t_1,\t_2}(u,v) \\
&+ \sum_{\nu=1}^{\infty} \frac{2}{(2\pi i)^2 a^{\frac{3}{2} \nu}} \sum_{\substack{k,l \geq 0\\ k+l = 2\nu -1}} \Re\left\{\hat{b}_{k,l} B_{k,l}\right\}\\
&+\sum_{\nu=1}^{\infty} \frac{2}{(2\pi i)^2 a^{\frac{3}{2} \nu}} \sum_{\substack{k,l \geq 0\\ k+l = \nu -1}} \Re\left\{\hat{c}_{2k,2l}e^{i\frac{4}{3} a^{\frac{3}{2}}}\right\}\Gamma\left(k+\frac{1}{2}\right)\Gamma\left(l +\frac{1}{2}\right),  
\end{align}
as \(a \to +\infty\), uniformly in \(u,v,\t_1, \t_2\), each coming from a compact subset of \(\mathbb{R}\). The coefficients \(B_{k,l}\) are defined in \eqref{def:B}, and the coefficients \(\hat{b}_{k,l}\) are given by \(b_{k,l}\) defined in \eqref{varphi_exp} after the substitutions \(u \to \pi u\), \(v\to \pi v\), \(\t_1 \to \frac{\pi^2 \t_1}{2}\) and \(\t_2 \to \frac{\pi^2 \t_2}{2}\), and the coefficients \(\hat{c}_{k,l}\) are given by \(c_{k,l}\) defined in \eqref{varPhi_exp} after the same substitutions.
\end{remark}
Let us also remark that convergence from the classical, i.e.~non-extended Airy kernel, to the non-extended sine kernel has been studied in \cite{AGK22}.\\

We now turn to the transition from the extended Pearcey kernel to the extended sine kernel. To simplify the presentation, we introduce a second variant of the sine kernel by
\begin{align}\label{S_2-kernel}
	\mathbb K^{\rm S_2}_{\t_1,\t_2}(u,v)&:=\frac{1}{2\pi i} \int_{e^{-i\pi/3}}^{e^{i\pi/3}}  ~ e^{(\t_1 - \t_2)\w^2 +(u-v)\w} d\w\\
	&-1(\t_1>\t_2)\frac{1}{\sqrt{4\pi(\t_1-\t_2)}}\exp\lb-\frac{(u-v)^2}{4(\t_1-\t_2)}\rb,
	\end{align}
where $\t_1,\t_2\in\R$ are time and $u,v\in\R$ are spatial arguments. It is easily checked that the \(\mathrm{S_1}\)- and \(\mathrm{S_2}\)-variants are related by
\begin{equation}\label{connSS} \frac{2}{\sqrt{3}} e^{\frac{1}{3} (\t_1 -\t_2) -\frac{1}{\sqrt{3}} (u-v)} \mathbb K^{\rm S_2}_{\frac{4}{3} \t_1, \frac{4}{3} \t_2}\left(\frac{2}{\sqrt{3}} u-\frac{4}{3}\t_1, \frac{2}{\sqrt{3}}v-\frac{4}{3}\t_2\right) = \mathbb K^{\rm S_1}_{\t_1,\t_2}(u,v),
\end{equation}
for \( u,v,\t_1,\t_2 \in \mathbb{R}\), which means that the \(\mathrm{S_2}\)-variant is just a rescaled and gauged version of the \(\rm S_1\)-variant, and thus of the extended sine kernel in \eqref{extended_sine}.
\begin{thm}\label{thm:Pearceytosine} 
	We have for any $u,v,\t_1, \t_2 \in\R$
	\begin{equation}\label{trans:PearceytoS}	\frac{1}{a^{1/3}} \mathbb  K^{\rm P}_{2\t_1/a^{2/3},2\t_2/a^{2/3}}\left(\frac{u}{a^{1/3}} +a, \frac{v}{a^{1/3}} +a \right)=\mathbb K^{ \mathrm S_2}_{\t_1,\t_2}(u,v) +  \mathrm{fluc^{\mathrm S_2}}(u,v; \t_1, \t_2; a) +\mathcal{O}\left(a^{- 8/3}\right) \end{equation}
as \(a\to+\infty\), uniformly in $u,v,\t_1, \t_2 $ when restricted to compact subsets of \(\mathbb{R}\),  where the fluctuations are given by
\begin{align}\mathrm{fluc^{\mathrm{S}_2}}(u,v; \t_1, \t_2; a) =& -\frac{2}{3}\frac{\pi}{(2\pi i)^2} \frac{\exp\left\{\frac{u-v}{2}-\frac{\t_1 -\t_2}{2}\right\}}{a^{\frac{4}{3}}} \\
& \times \left\{ f_{\tau_1, \tau_2}^{\mathrm{S}_2} (u,v)- \frac{2}{\sqrt{3}}  \cos\left(\frac{3\sqrt{3}}{4} a^{\frac{4}{3}} -\frac{u+v}{2}-\frac{\t_1 +\t_2}{2}\right) \right\}
\end{align}
with 
\begin{align}f_{\tau_1, \tau_2}^{\mathrm{S}_2} (u,v) &= \bigg\{\left( \frac{u+v}{2} -(\t_1 + \t_2)\right)\sin\left(\frac{\sqrt{3}}{2} (u-v+\t_1 -\t_2)\right)\\
 &+ \sqrt{3}\left( \frac{u+v}{2}+ (\t_1 + \t_2)\right)\cos\left(\frac{\sqrt{3}}{2} (u-v+\t_1 -\t_2)\right)\bigg\}.
\end{align}
More generally, for the Pearcey to sine kernel transition there exists a complete asymptotic expansion of the form

\begin{align}\label{Exp:PearceytoS} \frac{1}{a^{1/3}}& \mathbb  K^{\rm P}_{2\t_1/a^{2/3},2\t_2/a^{2/3}}\left(\frac{u}{a^{1/3}} +a, \frac{v}{a^{1/3}} +a \right)  \approx \mathbb K^{\rm S_2}_{\t_1,\t_2}(u,v) \\
&+ \sum_{\nu=1}^{\infty} \frac{2}{(2\pi i)^2 a^{\frac{4}{3} \nu}} \sum_{\substack{k,l \geq 0\\ k+l = 2\nu -1}} \Re\left\{b_{k,l} B_{k,l}\right\}\\
& +\sum_{\nu=1}^{\infty} \frac{2}{(2\pi i)^2 a^{\frac{4}{3} \nu}} \sum_{\substack{k,l \geq 0\\ k+l = \nu -1}} \Re\left\{c_{2k,2l} e^{-i \frac{3\sqrt{3}}{4} a^{\frac{4}{3}}}\right\}\Gamma\left(k+\frac{1}{2}\right)\Gamma\left(l +\frac{1}{2}\right),
\end{align}
as \(a \to +\infty\), uniformly in \(u,v,\t_1, \t_2\), each coming from a compact subset of \(\mathbb{R}\). The coefficients \(b_{k,l}\) are defined in \eqref{varphi_exp_P},  the coefficients \(c_{k,l}\) are defined in \eqref{varPhi_exp_P}, and the coefficients \(B_{k,l}\) are defined in \eqref{def:B}.
\end{thm}

\begin{remark} As in Remark \ref{remark1}, using relations \eqref{connSS} and \eqref{connS}, we can find suitable gauge factors and rescalings of the spatial and time variables, so that the leading terms in the asymptotic expression \eqref{trans:PearceytoS} and in the expansion \eqref{Exp:PearceytoS} become the extended sine kernel \(\mathbb K^{\rm sine}\) in \eqref{extended_sine}.
\end{remark}

Let us briefly comment on transitions between the three extended kernels other than the ones considered here.

It was shown in \cite{ACvM} with a PDE approach that (with a different parameterization)
\begin{align}
	&\lim_{a\to\infty}a\exp\lb{\frac1{27}a^{8/3}(\t_1-\t_2)+\frac{a^{4/3}}3(u-v)}\rb\\
	&\times\mathbb K^{\rm P}_{2a^{2/3}\t_1,2a^{2/3}\t_2}\lb a^{1/3}u+\frac23a^{5/3}\t_1-2\lb\frac a3\rb^3,a^{1/3}v+\frac23a^{5/3}\t_2-2\lb\frac a3\rb^3\rb\\
	&=\mathbb K^{\rm Ai}_{\t_1,\t_2}(u,v).\label{asymptoticrelation}
\end{align}
The transition from the non-extended Pearcey to Airy has also been studied for gap probabilities in \cite{BC}.
The transition from the extended Pearcey to the Airy kernel has been explained further in \cite{NV} by finding the \emph{Pearcey-to-Airy transition kernel with parameter $a$} given for $\t_1,\t_2,u,v\in\R$ and $a\geq0$ by  
\begin{align}\label{eq:transitionkernel}
	\mathbb K^a_{\t_1,\t_2}(u,v):&=\frac{1}{(2\pi i)^2} \int_{i\R} d\zeta \int_{\Gamma^{\rm P}} d\w~ \frac{\exp\lb-\frac{\zeta^4}{4}+\frac{a\z^3}{3}-\frac{\t_2\zeta^2}{2}-v\zeta +\frac{\w^4}{4}-\frac{a\w^3}{3}+\frac{\t_1\w^2}{2}+u\w\rb}{\zeta-\w}\\
	&-1(\t_1>\t_2)\frac{1}{\sqrt{2\pi(\t_1-\t_2)}}\exp\lb-\frac{(u-v)^2}{2(\t_1-\t_2)}\rb,
\end{align}
where the contour $\Gamma^{\rm P}$ is the X-shaped contour of the Pearcey kernel. This kernel interpolates between the Pearcey and the Airy kernel in the sense that $\mathbb K^0=\mathbb K^{\rm P}$ and for any $u,v,\t_1, \t_2 \in\R$
\begin{align}
	\lim_{a\to +\infty}a^{1/3}\mathbb K^a_{2a^{2/3}\t_1,2a^{2/3}\t_2}(a^{1/3}u,a^{1/3}v)=\mathbb K^{\rm Ai}_{\t_1,\t_2}(u,v). \label{interpol}
\end{align}
Shifting the integration variables $\zeta$ and $\w$ in \eqref{eq:transitionkernel} in such a way that the cubic terms are absorbed leads to an expression of the transition kernel in terms of the extended Pearcey kernel, and in turn to the relation \eqref{asymptoticrelation}. 

The transition kernel \eqref{eq:transitionkernel} arises asymptotically in NIBM when moving away from a cusp point at a certain speed in the Pearcey scaling (depending on the initial eigenvalues) \cite{NV}. We believe that such kernels interpolating between limiting kernels are easier to identify in the finite $n$ NIBM model rather than working with the limiting kernels directly. We may look at potential kernels interpolating between the extended Airy or Pearcey kernel, and the extended sine kernel, in the future.

Finally, let us comment on reversing the transitions. A transition from the extended Airy kernel to the extended Pearcey kernel has been provided in \cite{ADvM,AFvM} with the $r$-Airy kernel, interpolating from the extended Airy kernel ($r=0$) to the extended Pearcey kernel ($r\to\infty$). The $r$-Airy kernel arises as a limit in a variant of NIBM with Brownian bridges if $r$ eigenvalues are acting as outliers close to an edge point. As the number of outliers grows, a new bulk is formed leading to a cusp point with the first bulk, and hence the Pearcey kernel. It seems clear that to go from a translation-invariant extended sine kernel to the not translation-invariant extended Airy or Pearcey kernels, some modifications are needed to create an edge or a cusp out of a bulk. This is likely to be considered more naturally in models of a finite number of paths like NIBM.

\section{Proofs}

\vspace{2em}

\begin{proof}[Proof of Theorem \ref{thm:Airytosine}] For the heat kernel part in the definition of the Airy kernel \eqref{extended_Airy} we have, using the rescaling of the variables and time parameters given in \eqref{trans:AirytoS},
\[\frac{1}{\sqrt{a}}\frac{1}{\sqrt{4\pi(\t_1/a-\t_2/a)}}\exp\lb-\frac{(u/\sqrt{a}-v/\sqrt{a})^2}{4(\t_1/a-\t_2/a)}\rb=\frac{1}{\sqrt{4\pi(\t_1-\t_2)}}\exp\lb-\frac{(u-v)^2}{4(\t_1-\t_2)}\rb,\]
 which agrees with the heat kernel part in the definition of the \(\mathrm S_1\)-kernel \eqref{S_1-kernel}. Hence, we can focus on the double contour integral in \eqref{extended_Airy}, which we will denote by \(I_{\t_1, \t_2}(u,v)\). We have
\begin{align}&\frac{1}{\sqrt{a}}I_{\t_1/a, \t_2/a}\left(\frac{u}{\sqrt{a}}-a,\frac{v}{\sqrt{a}}-a \right) \\ 
&=\frac{1}{\sqrt{a}( 2\pi i)^2} \int_{\Sigma^{\rm Ai}} d\zeta \int_{\Gamma^{\rm Ai}} d\w ~ \frac{\exp\lb\frac{\zeta^3}3-\zeta \left(\frac{v}{\sqrt{a}}-a\right)-\frac{\t_2}{a}\zeta^2-\frac{\w^3}3+\w \left(\frac{u}{\sqrt{a}}-a\right)+\frac{\t_1}{a}\w^2\rb}{\zeta-\w}\\
&=\frac{1}{( 2\pi i)^2}  \int_{\Sigma^{\rm Ai}} d\zeta \int_{\Gamma^{\rm Ai}} d\w ~\frac{\exp\lb a^{3/2}\left(\frac{\zeta^3}{3}+\zeta\right) - a^{3/2}\left(\frac{\w^3}{3}+\w\right) -v\zeta -\t_2 \zeta^2 + u \w + \t_1 \w^2\rb}{\zeta-\w},
\end{align}
 where in the last step we have rescaled \(\zeta \) and \(\w\) by \(\sqrt{a}\). For the next step, we recall that the integration contour for \(\zeta\) consists of a straight line connecting \(\infty e^{- i\pi/3 }\) to the origin followed by another straight line connecting the origin to  \(\infty e^{+ i\pi/3 }\). Likewise, the integration contour for \(\w\) consists of a straight line connecting \(\infty e^{- i2\pi/3 }\) to the origin followed by another straight line connecting the origin to  \(\infty e^{+ i 2\pi/3 }\). We now shift the \(\zeta\) contour horizontally to the left and the \(\w\) contour horizontally to the right so that the contours meet at the points \(i\) and \(-i\). The new \(\zeta\) contour we denote by \(\tilde{\Sigma}^{\rm Ai}\) and the new \(\w\) contour we denote by \(\tilde{\Gamma}^{\rm Ai}\), see Figure \ref{Airy_contours_shifted}. Then we have

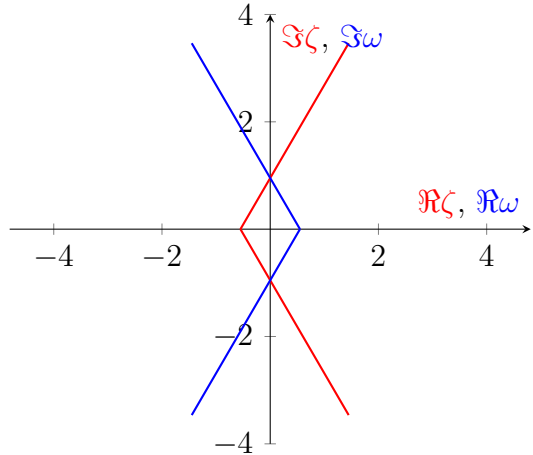
\begin{figure}[h]
\centering
\begin{tikzpicture}
    \begin{axis}[
        axis lines=middle,
        xlabel={$\textcolor{red}{\Re \zeta}$, $\textcolor{blue}{\Re \omega}$ },
        ylabel={$\textcolor{red}{\Im \zeta}$, $\textcolor{blue}{\Im \omega}$  },
        xmin=-4, xmax=4,
        ymin=-4, ymax=4,
        domain=-4:4,
        samples=100,
        axis equal,
        xtick={-4,-2,0,2,4},
        ytick={-4,-2,0,2,4}
    ]
    \addplot [thick, red, domain=0:4, samples=100] ({x*cos(60) - 0.55}, {x*sin(60)});
    \addplot [thick, red, domain=0:4, samples=100] ({x*cos(-60) - 0.55}, {x*sin(-60)});
    
    \addplot [thick, blue, domain=0:4, samples=100] ({x*cos(120) + 0.55}, {x*sin(120)});
    \addplot [thick, blue, domain=0:4, samples=100] ({x*cos(-120) + 0.55}, {x*sin(-120)});
    
    \end{axis}
\end{tikzpicture}

\caption{The contours \(\tilde{\Sigma}^{\rm Ai}\) and \(\tilde{\Gamma}^{\rm Ai}\) }
\label{Airy_contours_shifted}
\end{figure}

\begin{align}&\frac{1}{\sqrt{a}}I_{\t_1/a, \t_2/a}\left(\frac{u}{\sqrt{a}}-a,\frac{v}{\sqrt{a}}-a \right) \\ \label{int1}
&=\frac{1}{( 2\pi i)^2}  \int_{\tilde{\Sigma}^{\rm Ai}} d\zeta \int_{\tilde{\Gamma}^{\rm Ai}} d\w ~\frac{\exp\lb a^{3/2}\left(\frac{\zeta^3}{3}+\zeta\right) - a^{3/2}\left(\frac{\w^3}{3}+\w\right) -v\zeta -\t_2 \zeta^2 + u \w + \t_1 \w^2\rb}{\zeta-\w}\\
&\quad\quad +  \frac{1}{2\pi i} \int_{-i}^{i}  ~ e^{(\t_1 - \t_2)\w^2 +(u-v)\w} d\w,
\end{align}
where the single integral arises from an application of the residue theorem and Cauchy's integral theorem. Indeed, for any holomorphic function \(F(\zeta, \w)\) such that the following double integral exists
\[\frac{1}{( 2\pi i)^2}  \int_{\Sigma^{}} d\zeta \int_{\Gamma^{}} d\w ~ \frac{F(\zeta, \w)}{\zeta -\w  }, \]
 it can be shown using the residue theorem, that we have 
\begin{align}\label{intersect}\frac{1}{( 2\pi i)^2}  \int_{\Sigma^{}} d\zeta \int_{\Gamma^{}} d\w ~ \frac{F(\zeta, \w)}{\zeta -\w  } &= \frac{1}{( 2\pi i)^2}  \int_{\tilde{\Sigma}^{}} d\zeta \int_{\tilde{\Gamma}^{}} d\w ~  \frac{F(\zeta, \w)}{\zeta -\w  }\\
&\quad + \frac{1}{2\pi i} \int_{T_1}^{T_2} F(\w, \w) d\w,
\end{align}
where the contours \(\tilde{\Sigma}\) and \(\tilde{\Gamma}\) are modifications of \(\Sigma\) and \(\Gamma\) in the above described way, intersecting exactly at the points \(T_1 \) and \(T_2\). More precisely, pushing the contours into each other creates two integrable singularities at the points of intersection, so we have to take into account the contributions coming from these points. It follows from this representation, that the derivation of \eqref{trans:AirytoS} is based in an asymptotic evaluation for large \(a\to+\infty\) of the double contour integral in \eqref{int1}
\begin{align}
	&J_{\t_1, \t_2}^{a}(u,v)\\
	&:=\frac{1}{( 2\pi i)^2}  \int_{\tilde{\Sigma}^{\rm Ai}} d\zeta \int_{\tilde{\Gamma}^{\rm Ai}} d\w ~\frac{\exp\lb a^{3/2}\left(\frac{\zeta^3}{3}+\zeta\right) - a^{3/2}\left(\frac{\w^3}{3}+\w\right) -v\zeta -\t_2 \zeta^2 + u \w + \t_1 \w^2\rb}{\zeta-\w}.\label{def:J}
\end{align}
 We aim for a complete asymptotic expansion for \(J_{\t_1, \t_2}^{a}(u,v)\), which in turn gives a complete expansion for the rescaled Airy kernel in \eqref{trans:AirytoS} with the \(\mathrm{S}_1\)-kernel as its leading term. We introduce the auxiliary function \(f(\z) = \frac{\z^3}{3} +\z\), so that the integral \(J_{\t_1, \t_2}^{a}(u,v)\) is given by
\begin{align}\label{eq:J1} \frac{1}{( 2\pi i)^2}  \int_{\tilde{\Sigma}^{\rm Ai}} d\zeta \int_{\tilde{\Gamma}^{\rm Ai}} d\w ~ \exp\left\{ a^{3/2}\left( f(\z) - f(\w)\right)\right\}\frac{\exp\lb  -v\zeta -\t_2 \zeta^2 + u \w + \t_1 \w^2\rb}{\zeta-\w}.
\end{align}
Without difficulty we compute that there are four two-dimensional (simple) saddle points of the phase function \(f(\z) - f(\w)\), which are given by \(\left\{ (i,i), (i,-i), (-i,i), (-i,-i) \right\}\). We will study the contributions from each of these saddle points seperately, and construct the asymptotic expansion from these parts. From the perspective of asymptotic theory it is interesting to remark that  \(\left\{ (i,i), (-i,-i) \right\}\) and   \(\left\{ (i,-i), (-i,i) \right\}\) form two groups of saddle points that do not lead to the same type of asymptotic behavior. The reason for this is that due to the geometry of the integrand in \eqref{eq:J1}, at each of the saddle points in \(\left\{ (i,i), (-i,-i) \right\}\) the integrand exhibits an integrable singularity, which does not happen in the case of the remaining two saddle points. We have already arranged that the contours pass through the saddle points. However, in order to carry out a saddle point analysis, we now deform the contours \(\tilde{\Sigma}^{\rm Ai}\) and \(\tilde{\Gamma}^{\rm Ai}\) into paths of steepest descent. In general, the paths of steepest descent and steepest ascent are characterized by the property that the imaginary part of the phase function remains constant. As in our case the phase function has the specific form \(f(\z) - f(\w)\), we can make use of the following observation: the steepest descent path for \(f(\z)\) passing through the critical point \(i\) is the steepest ascent path for \(-f(\w)\) though \(i\), and the steepest descent path for \(-f(\w)\) is the steepest ascent path for \(f(\z)\). The two paths passing through the point \(i\) are given by the equation \(\Im\{\frac{1}{3} \z^3 + \z \}= \frac{2}{3}\), see Figure \ref{fig:contour_Airy1}. It can be easily checked that the steepest descent path for \(f(\z)\) is given by the branch that starts on the negative real line at \(-\infty\), passes through \(i\) and runs off to \(+\infty e^{i\pi/3}\), so that along this path, the function \(f(\z)-i \frac{2}{3}\) runs from \(-\infty\) to 0, and returns back to \(-\infty\). We denote this steepest descent path for \(f(\z)\) in the following by \(S_+\). The second contour in Figure \ref{fig:contour_Airy1} is the steepest descent path of \(-f(\w)\), it starts on the positive real axis at \(+\infty\), passes through \(i\) and runs off to  \(+\infty e^{i 2\pi/3}\). We will denote this path by \(T_+\). The situation is similar for the critical point \(-i\). In this case, the steepest descent and steepest ascent paths for \(f(\z)\) through \(-i\) are characterized by the equation  \(\Im\{\frac{1}{3} \z^3 + \z \}= -\frac{2}{3}\), see Figure \ref{fig:contour_Airy2}.The steepest descent path starts at \(+\infty e^{-i\pi/3}\), passes through \(-i\) and runs off to \(-\infty\), which we will denote by \(S_-\). Likewise, the second branch in Figure \ref{fig:contour_Airy2} is the steepest descent path for \(-f(\w)\), it starts at \(+\infty e^{-2\pi/3}\), passes through \(-i\) and runs off to \(+\infty\), and we will denote it by \(T_-\).

\begin{figure}[h]
    \centering
    \begin{tikzpicture}
        \begin{axis}[
            axis equal,
            xlabel={$\Re{\z}$},
            ylabel={$\Im{\z}$},
            ]
            \addplot [
                only marks,
                mark=*, 
                mark size=0.3pt,
                ] table {contour_data_Airy.txt};
        \end{axis}
    \end{tikzpicture}
    \caption{Contour plot of $\Im\{\frac{1}{3} \z^3 + \z \}= \frac{2}{3}$ in the upper complex plane}
    \label{fig:contour_Airy1}
\end{figure}
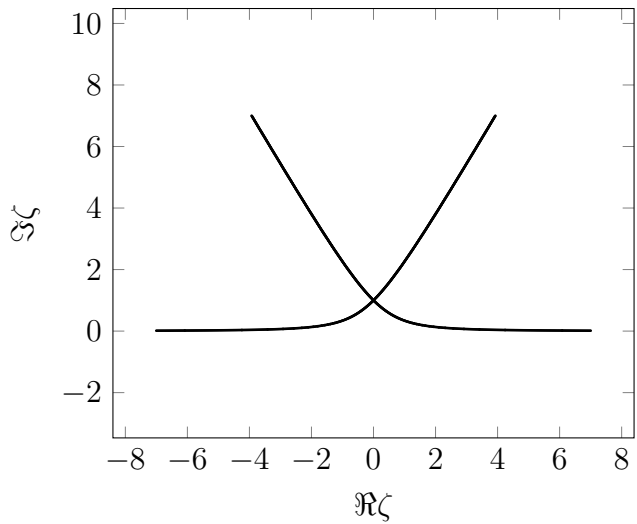

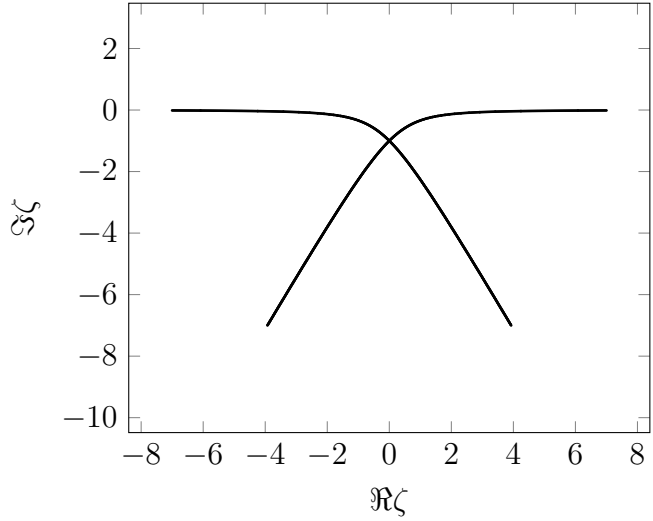
\begin{figure}[h]
    \centering
    \begin{tikzpicture}
        \begin{axis}[
            axis equal,
            xlabel={$\Re{\z}$},
            ylabel={$\Im{\z}$},
            ]
            \addplot [
                only marks,
                mark=*, 
                mark size=0.3pt,
                ] table {contour_data_Airy2.txt};
        \end{axis}
    \end{tikzpicture}
    \caption{Contour plot of $\Im\{\frac{1}{3} \z^3 + \z \}= -\frac{2}{3}$ in the lower complex plane}
    \label{fig:contour_Airy2}
\end{figure}

We now use Cauchy's theorem in order to deform the contours of integration in the representation \eqref{eq:J1} into the just found steepest descent contours: the contour \(\tilde{\Sigma}^{\rm Ai}\) (the red contour in Figure \ref{Airy_contours_shifted}) we deform into the concatenation of the contours \(S_-\) and \(S_+\), which we simply call \(\mathrm{S}\). Analogously, we deform \(\tilde{\Gamma}^{\rm Ai}\) (the blue contour in Figure \ref{Airy_contours_shifted}) into the concatenation of \(T_-\) and \(T_+\), and call it \(\mathrm{T}\). This gives us for \(J_{\t_1, \t_2}^{a}(u,v)\) the expression

\begin{align}\label{eq:J2} \frac{1}{( 2\pi i)^2}  \int_{\mathrm{S}} d\zeta \int_{\mathrm{T}} d\w ~ \exp\left\{ a^{3/2}\left( f(\z) - f(\w)\right)\right\}\frac{\exp\lb  -v\zeta -\t_2 \zeta^2 + u \w + \t_1 \w^2\rb}{\zeta-\w}.
\end{align}
In this setup, we now are able to study the asymptotic contributions from each of the four saddle points.
\subsection{Contribution from \((i,i)\)}

To find the contribution of the saddle point \((i,i)\), we only need to focus on a small neighborhood of the point \((i,i)\). More precisely, we consider the integral
\begin{align}\label{eq:S1} \frac{1}{( 2\pi i)^2}  \int_{\mathrm{S}(i)} d\zeta \int_{\mathrm{T}(i)} d\w ~ \exp\left\{ a^{3/2}\left( f(\z) - f(\w)\right)\right\}\frac{\exp\lb  -v\zeta -\t_2 \zeta^2 + u \w + \t_1 \w^2\rb}{\zeta-\w},
\end{align}
where \(\mathrm{S}(i)\) and \(\mathrm{T}(i)\) are the parts of the steepest descent paths \(\mathrm{S}\) and \(\mathrm{T}\) that intersect with a small complex neighborhood of \(i\). As the integration contour \(\mathrm{S}(i)\) is the steepest descent path for \(f(\z) = \frac{\z^3}{3}+ \z\) with climax at the point \(\z =i\), and \(\z=i\) is a simple critical point of \(f\), we can find a conformal analytic function \(g\) defined on some \(\epsilon\)-ball centered at the origin \(U_{\e}(0)\), mapping the interval \((-\e, \e) \subset \mathbb{R}\) bijectively onto \(\mathrm{S}(i)\) such that
\begin{equation}\label{algequ} f(g(x))-\frac{2}{3}i = -x^2, \quad -\e < x < \e.\end{equation}
More precisely, the function \(g\) is the branch of the algebraic function \(g\) defined by \eqref{algequ} with the series expansion
\begin{align}\label{g_exp} g(x)= \sum_{k=0}^{\infty} a_k x^k,
\end{align}
the coefficients of which can be obtained recursively from equation \eqref{algequ}. The first coefficients are given by \(a_0= i, a_1 = \sqrt{i}, a_2 = -\frac{1}{6}, a_3 = \frac{5}{72 \sqrt{i}}.\) It follows by replacing \(x\) by \(iy\) in \eqref{algequ} that we have
\begin{equation}\label{algequ2} f(g(iy))-\frac{2}{3}i = y^2, \quad -\e < y < \e,\end{equation}
which shows that \(y\mapsto g(iy)\) is a conformal map from a neighborhood of the origin to a neighborhood of the point \(i\), mapping the interval \((-\e, \e)\) onto \(\mathrm{T}(i)\).
In the integral \eqref{eq:S1}, we perform the change of variables \(\z = g(x)\) and \(\w = g(iy)\), so that we get
\begin{align}\label{eq:S2} \frac{1}{( 2\pi i)^2}  \int_{-\e}^{\e} dx \int_{-\e}^{\e} dy ~ \exp\left\{- a^{3/2}\left( x^2 +y^2\right)\right\}\frac{\varphi(x,y)}{x-iy},
\end{align}
where the function \(\varphi(x,y)\) is given by
\begin{align}\label{varphi} \varphi(x,y) = \frac{x-iy}{g(x)-g(iy)}\exp\left\{ -vg(x) -\t_2 (g(x))^2+u g(iy) +\t_1 (g(iy))^2 \right\}\frac{d}{dx}g(x) \frac{d}{dy}g(iy).
\end{align}
The function \(\varphi(x,y) \) is an analytic function of two complex variables in a neighborhood of the origin \((0,0)\), so we can expand it into a power series of the form
\begin{equation}\label{varphi_exp}\varphi(x,y) = \sum_{k,l \geq 0}b_{k,l} x^k y^l,\end{equation}
in which the coefficients \(b_{k,l}\) depend on the variables \(u,v\) and parameters \(\t_1, \t_2\). At this point, we have to choose the size of \(\e\) small enough, so that the power series in \eqref{varphi_exp} converges (absolutely) in some polydisc containing \((-\e, \e)^2\). We can choose this \(\e\) independently of the variables \(u,v\) and parameters \(\t_1, \t_2\) when we restrict these quantities to come from a compact set. After substituting \eqref{varphi_exp} in \eqref{eq:S2}, we obtain a complete asymptotic expansion for the contribution of \((i,i)\) by termwise integration. More precisely, we have

\begin{align} \frac{1}{( 2\pi i)^2} &  \int_{-\e}^{\e} dx \int_{-\e}^{\e} dy ~ \exp\left\{- a^{3/2}\left( x^2 +y^2\right)\right\}\frac{\varphi(x,y)}{x-iy}\\
\approx & \sum_{k,l \geq 0} \frac{b_{k,l}}{a^{\frac{3}{4} (k+l+1)}}  \frac{1}{( 2\pi i)^2}  \int_{-\e a^{\frac{3}{4}}}^{\e a^{\frac{3}{4}}} dx \int_{-\e a^{\frac{3}{4}}}^{\e a^{\frac{3}{4}}} dy ~ \exp\left\{- \left( x^2 +y^2\right)\right\}\frac{x^k y^l}{x-iy},\label{eq:S3}
\end{align}
where the symbol \(\approx\) in \eqref{eq:S3} means that, for every integer \(N \geq 0\), we have
\begin{align} \frac{1}{( 2\pi i)^2} &  \int_{-\e}^{\e} dx \int_{-\e}^{\e} dy ~ \exp\left\{- a^{3/2}\left( x^2 +y^2\right)\right\}\frac{\varphi(x,y)}{x-iy}\\
= & \sum_{k,l \geq 0 , ~ k+l \leq N} \frac{b_{k,l}}{a^{\frac{3}{4} (k+l+1)}}  \frac{1}{( 2\pi i)^2}  \int_{-\e a^{\frac{3}{4}}}^{\e a^{\frac{3}{4}}} dx \int_{-\e a^{\frac{3}{4}}}^{\e a^{\frac{3}{4}}} dy ~ \exp\left\{- \left( x^2 +y^2\right)\right\}\frac{x^k y^l}{x-iy}\\
+& \mathcal{O}\left(\frac{1}{a^{\frac{3}{4} (N+2)}}\right), \label{eq:SS3}
\end{align}
as \(a\to\infty\), uniformly in \(u,v,\t_1,\t_2\) on compacts. Indeed, we can estimate
\begin{align} & \left\vert  \frac{a^{\frac{3}{4}(N+2)}}{( 2\pi i)^2}  \int_{-\e a^{\frac{3}{4}}}^{\e a^{\frac{3}{4}}} dx \int_{-\e a^{\frac{3}{4}}}^{\e a^{\frac{3}{4}}} dy ~ \frac{\exp\left\{- \left( x^2 +y^2\right)\right\}}{x-iy} \sum_{k,l \geq 0 , ~ k+l \geq N+1} \frac{b_{k,l}}{a^{\frac{3}{4} (k+l+1)}} x^k y^l  \right\vert \\
& \leq \frac{1}{(2\pi)^2} \int_{-\e a^{\frac{3}{4}}}^{\e a^{\frac{3}{4}}} dx \int_{-\e a^{\frac{3}{4}}}^{\e a^{\frac{3}{4}}} dy ~ \frac{\exp\left\{- \left( x^2 +y^2\right)\right\}}{\sqrt{x^2 +y^2}} a^{\frac{3}{4}(N+1)} \sum_{k,l \geq 0 , ~ k+l \geq N+1} \vert b_{k,l}\vert \left(\frac{\vert x\vert}{a^{\frac{3}{4}}}\right)^k \left(\frac{\vert y\vert}{a^{\frac{3}{4}}}\right)^l,  \label{eq:SS4}
\end{align}
where the inner sum can be written as
\begin{align} & a^{\frac{3}{4}(N+1)} \sum_{k,l \geq 0 , ~ k+l \geq N+1} \vert b_{k,l}\vert \left(\frac{\vert x\vert}{a^{\frac{3}{4}}}\right)^k \left(\frac{\vert y\vert}{a^{\frac{3}{4}}}\right)^l \\
&= a^{\frac{3}{4}(N+1)} \sum_{k=0}^{N+1} \sum_{l\geq N+1 - k} \vert b_{k,l}\vert \left(\frac{\vert x\vert}{a^{\frac{3}{4}}}\right)^k \left(\frac{\vert y\vert}{a^{\frac{3}{4}}}\right)^l + a^{\frac{3}{4}(N+1)}  \sum_{k \geq N+2} \sum_{l \geq 0}  \vert b_{k,l}\vert \left(\frac{\vert x\vert}{a^{\frac{3}{4}}}\right)^k \left(\frac{\vert y\vert}{a^{\frac{3}{4}}}\right)^l.
\end{align}
Further estimating these sums gives for the first one
\begin{align}& a^{\frac{3}{4}(N+1)} \sum_{k=0}^{N+1} \sum_{l\geq N+1 - k} \vert b_{k,l}\vert \left(\frac{\vert x\vert}{a^{\frac{3}{4}}}\right)^k \left(\frac{\vert y\vert}{a^{\frac{3}{4}}}\right)^l = \sum_{k=0}^{N+1} \vert x \vert ^k \vert y\vert^{N+1-k}  \sum_{l\geq N+1 - k}  \vert b_{k,l}\vert  \left(\frac{\vert y\vert}{a^{\frac{3}{4}}}\right)^{l-(N+1-k)} \\
&\leq  \sum_{k=0}^{N+1} \vert x \vert ^k \vert y\vert^{N+1-k}  \sum_{l\geq N+1 - k}  \vert b_{k,l}\vert  \e^{l-(N+1-k)} \leq K  \sum_{k=0}^{N+1} \vert x \vert ^k \vert y\vert^{N+1-k} ,
\end{align}
for some positive constant \(K>0\) independent of \(u,v,\t_1, \t_2\), and for the second one
\begin{align}&a^{\frac{3}{4}(N+1)}  \sum_{k \geq N+2} \sum_{l \geq 0}  \vert b_{k,l}\vert \left(\frac{\vert x\vert}{a^{\frac{3}{4}}}\right)^k \left(\frac{\vert y\vert}{a^{\frac{3}{4}}}\right)^l= \frac{\vert x \vert ^{N+2}}{a^{\frac{3}{4}}}  \sum_{k \geq N+2} \sum_{l \geq 0}  \vert b_{k,l}\vert \left(\frac{\vert x\vert}{a^{\frac{3}{4}}}\right)^{k-(N+2)} \left(\frac{\vert y\vert}{a^{\frac{3}{4}}}\right)^l \\
&\leq \frac{\vert x \vert ^{N+2}}{a^{\frac{3}{4}}} \sum_{k \geq N+2} \sum_{l \geq 0}  \vert b_{k,l}\vert \e^{k-(N+2)}\e^l \leq K' \vert x \vert ^{N+2},
\end{align}
for some positive constant \(K'>0\) independent of \(u,v,\t_1, \t_2\). Using these estimates in \eqref{eq:SS4} and observing that all involved integrals converge absolutely as \(a\to\infty\), we obtain the statement in \eqref{eq:SS3}.
Moreover, it is not difficult to observe that we have for every \(k,l\geq 0\)

\begin{equation}\label{eq:S4}  \int_{-\e a^{\frac{3}{4}}}^{\e a^{\frac{3}{4}}} dx \int_{-\e a^{\frac{3}{4}}}^{\e a^{\frac{3}{4}}} dy ~ \exp\left\{- \left( x^2 +y^2\right)\right\}\frac{x^k y^l}{x-iy} = B_{k,l} + \mathcal{O}\left(\left(\e a^{\frac{3}{4}}\right)^{k+l-1} e^{-\e^2 a^{\frac{3}{2}}}\right),
\end{equation}
as \(a \to +\infty\), where we define
\begin{equation}\label{def:B} B_{k,l} =  \int_{-\infty}^{\infty} dx \int_{-\infty}^{\infty} dy ~ \exp\left\{- \left( x^2 +y^2\right)\right\}\frac{x^k y^l}{x-iy}, \quad k,l \geq 0.
\end{equation}
We remark that these coefficients \(B_{k,l}\) can be represented as 
\[B_{k,l} = \frac{1}{2} \Gamma\left(\frac{k+l+1}{2}\right) \int_{-\pi}^{\pi} e^{i\phi} \cos(\phi)^k  \sin(\phi)^l d\phi,\]
from which we could deduce a recursion, and that \(B_{k,l}= 0\) if and only if \(k \) and \(l\) are both even or both odd. The first coefficients are given by
\begin{equation} \label{Bs} B_{0,0} = 0, \quad B_{1,0} = \frac{\pi}{2},\quad B_{0,1}= i \frac{\pi}{2}.
\end{equation}
Substituting \eqref{eq:S4} into \eqref{eq:S3} shows now that the contribution from the saddle point \((i,i)\) to  \(J_{\t_1, \t_2}^{a}(u,v)\) is given by the asymptotic expansion

\begin{align}& \frac{1}{( 2\pi i)^2}  \int_{\mathrm{S}(i)} d\zeta \int_{\mathrm{T}(i)} d\w ~ \exp\left\{ a^{3/2}\left( f(\z) - f(\w)\right)\right\}\frac{\exp\lb  -v\zeta -\t_2 \zeta^2 + u \w + \t_1 \w^2\rb}{\zeta-\w}\\
&\approx  \sum_{k,l \geq 0} \frac{b_{k,l}}{a^{\frac{3}{4} (k+l+1)}}  \frac{B_{k,l}}{( 2\pi i)^2}   \label{eq:SS5}\\
&= \sum_{k,l \geq 0, ~ k+l ~\text{odd}} \frac{b_{k,l}}{a^{\frac{3}{4} (k+l+1)}}  \frac{B_{k,l}}{( 2\pi i)^2} = \sum_{\nu=1}^{\infty} \frac{1}{(2\pi i)^2 a^{\frac{3}{2} \nu}} \sum_{\substack{k,l \geq 0\\ k+l = 2\nu -1}} b_{k,l} B_{k,l} \label{eq:S5}
\end{align}
as \(a \to +\infty\), uniformly in \(u,v,\t_1, \t_2\), each coming from a compact subset of \(\mathbb{R}\). The equalities in \eqref{eq:S5} mean identity in the sense of asymptotic expansions, as the series are not convergent.

\subsection{Contribution from \((i, -i)\)} 

In this case we focus on a small neighborhood of the point \((i, -i)\). We consider the integral
\begin{align}\label{eq:S6} &\frac{1}{( 2\pi i)^2}  \int_{\mathrm{S}(i)} d\zeta \int_{\mathrm{T}(-i)} d\w ~ \exp\left\{ a^{3/2}\left( f(\z) - f(\w)\right)\right\}\frac{\exp\lb  -v\zeta -\t_2 \zeta^2 + u \w + \t_1 \w^2\rb}{\zeta-\w}\qquad\\
& = \frac{e^{i \frac{4}{3}a^{\frac{3}{2}}}}{( 2\pi i)^2}  \int_{\mathrm{S}(i)} d\zeta \int_{\mathrm{T}(-i)} d\w ~ \exp\left\{ a^{3/2}\left( f(\z)-\frac{2}{3}i\right) -a^{\frac{2}{3}}\left( f(\w)+\frac{2}{3}i\right)\right\}\\
&\qquad\qquad\times \frac{\exp\lb  -v\zeta -\t_2 \zeta^2 + u \w + \t_1 \w^2\rb}{\zeta-\w},
\end{align}
where \(\mathrm{S}(i)\) and \(\mathrm{T}(-i)\) are the parts of the steepest descent paths \(\mathrm{S}\) and \(\mathrm{T}\) that intersect with a small complex neighborhood of \(i\) and \(-i\), respectively. In order to construct the appropriate local transformations of the integration variables, we again make use of the function \(g\) defined in \eqref{algequ} satisfying
\[f(g(x))-\frac{2}{3}i = -x^2 , \quad -\e <x < \e,\]
which conformally maps a neighborhood of the origin to a neighborhood of \(i\). 
Moreover, we have 
\[f(-g(-y))+\frac{2}{3}i = y^2, \quad -\e < y < \e,\]
which shows that \(y \mapsto -g(-y)\) is a conformal map from a neighborhood of the origin to a neighborhood of the point \(-i\), mapping the interval \((-\e, \e)\) onto \(\mathrm{T}(-i)\). in the integral \eqref{eq:S6} we now perform the change of variables \(\z = g(x)\) and \(\w = -g(-y)\), so that we get
\begin{align}\label{eq:S7} \frac{e^{i\frac{4}{3} a^{\frac{3}{2}}}}{( 2\pi i)^2}  \int_{-\e}^{\e} dx \int_{-\e}^{\e} dy ~ \exp\left\{- a^{3/2}\left( x^2 +y^2\right)\right\}\varPhi(x,y),
\end{align}
where the function \(\varPhi(x,y)\) is given by
\begin{align}
	\varPhi(x,y) := \frac{\exp\left\{ -vg(x) -\t_2 (g(x))^2-u g(-y) +\t_1 (g(-y))^2 \right\}}{g(x)+g(-y)}\frac{d}{dx}g(x) \frac{d}{dy}\left(-g(-y)\right).\qquad  \label{varPhi} 
\end{align}
The function \(\varPhi(x,y) \) is an analytic function of two complex variables in a neighborhood of the origin \((0,0)\), so we can expand it into a power series of the form
\begin{equation}\label{varPhi_exp}\varPhi(x,y) = \sum_{k,l \geq 0}c_{k,l} x^k y^l,\end{equation}
in which the coefficients \(c_{k,l}\) depend on the variables \(u,v\) and parameters \(\t_1, \t_2\). Again here, we have to choose the size of \(\e\) small enough, so that the power series in \eqref{varPhi_exp} converges (absolutely) in some polydisc containing \((-\e, \e)^2\), independently of the variables \(u,v\) and parameters \(\t_1, \t_2\), each coming from a compact set.  After substituting \eqref{varPhi_exp} in \eqref{eq:S7} and after termwise integration, we obtain for the contribution of \((i,-i)\) the asymptotic expansion
\begin{align}\label{eq:S8}\sum_{k,l \geq 0} \frac{c_{k,l}}{a^{\frac{3}{4} (k+l+2)}} \frac{e^{i\frac{4}{3} a^{\frac{3}{2}}}}{( 2\pi i)^2}  \int_{-\e a^{\frac{3}{4}}}^{\e a^{\frac{3}{4}}} dx  \int_{-\e a^{\frac{3}{4}}}^{\e a^{\frac{3}{4}}} dy ~ \exp\left\{- \left( x^2 +y^2\right)\right\}x^k y^l,
\end{align}
as \(a \to +\infty\), uniformly in \(u,v,\t_1, \t_2\), each coming from a compact subset of \(\mathbb{R}\).
The double integral in \eqref{eq:S8} decouples, and we obtain for \(k,l \geq 0\)
\begin{align}\label{eq:S9}  \int_{-\e a^{\frac{3}{4}}}^{\e a^{\frac{3}{4}}} dx  \int_{-\e a^{\frac{3}{4}}}^{\e a^{\frac{3}{4}}} dy ~ \exp\left\{- a^{3/2}\left( x^2 +y^2\right)\right\}x^k y^l = C_k C_l +\mathcal{O}\left(e^{-\e^2 a^{\frac{3}{2}}}\right),
\end{align}
as \(a\to +\infty\), where the quantities \(C_k\) are the rescaled moments of a Gaussian distribution defined by
\begin{align}\label{def:C} C_k = \int_{-\infty}^{\infty} e^{-x^2} x^k dx = \frac{1+(-1)^k}{2} \Gamma\left(\frac{k+1}{2}\right), \quad k \geq 0.
\end{align}
Substituting \eqref{eq:S9} into \eqref{eq:S8} we obtain that the contribution of the saddle point \((i,-i)\) to \(J_{\t_1, \t_2}^{a}(u,v)\) is given by the asymptotic expansion
\begin{align}&\frac{1}{( 2\pi i)^2}  \int_{\mathrm{S}(i)} d\zeta \int_{\mathrm{T}(-i)} d\w ~ \exp\left\{ a^{3/2}\left( f(\z) - f(\w)\right)\right\}\frac{\exp\lb  -v\zeta -\t_2 \zeta^2 + u \w + \t_1 \w^2\rb}{\zeta-\w}  \\
&\approx\sum_{k,l \geq 0} \frac{c_{k,l} e^{i\frac{4}{3} a^{\frac{3}{2}}}}{a^{\frac{3}{4} (k+l+2)}} \frac{  C_k C_l }{( 2\pi i)^2} = \sum_{k,l \geq 0} \frac{c_{2k,2l} e^{i\frac{4}{3} a^{\frac{3}{2}}}}{a^{\frac{3}{2} (k+l+1)}} \frac{ \Gamma(k+\frac{1}{2})\Gamma(l+\frac{1}{2})  }{( 2\pi i)^2}  \\
&=   \sum_{\nu=1}^{\infty} \frac{1}{(2\pi i)^2 a^{\frac{3}{2} \nu}} \sum_{\substack{k,l \geq 0\\ k+l = \nu -1}} c_{2k,2l}e^{i\frac{4}{3} a^{\frac{3}{2}}}\Gamma\left(k+\frac{1}{2}\right)\Gamma\left(l +\frac{1}{2}\right),\label{eq:S10}
\end{align}
as \(a\to +\infty\), uniformly in \(u,v,\t_1, \t_2\), each coming from a compact subset of \(\mathbb{R}\).

\subsection{Contribution from \((-i, i)\)} Next, we focus on a small neighborhood of  \((-i,i)\).  We consider the integral
\begin{align} &\frac{1}{( 2\pi i)^2}  \int_{\mathrm{S}(-i)} d\zeta \int_{\mathrm{T}(i)} d\w ~ \exp\left\{ a^{3/2}\left( f(\z) - f(\w)\right)\right\}\frac{\exp\lb  -v\zeta -\t_2 \zeta^2 + u \w + \t_1 \w^2\rb}{\zeta-\w}\\
& = \frac{e^{-i \frac{4}{3}a^{\frac{3}{2}}}}{( 2\pi i)^2}  \int_{\mathrm{S}(-i)} d\zeta \int_{\mathrm{T}(i)} d\w ~ \exp\left\{ a^{3/2}\left( f(\z)+\frac{2}{3}i\right) -a^{\frac{2}{3}}\left( f(\w)-\frac{2}{3}i\right)\right\}\\
&\qquad\qquad\times \frac{\exp\lb  -v\zeta -\t_2 \zeta^2 + u \w + \t_1 \w^2\rb}{\zeta-\w}, \label{eq:S11}
\end{align}
where \(\mathrm{S}(-i)\) and \(\mathrm{T}(i)\) are the parts of the steepest descent paths \(\mathrm{S}\) and \(\mathrm{T}\) that intersect with a small complex neighborhood of \(-i\) and \(i\), respectively. As previously, we use the function \(g\) defined in \eqref{algequ} in order to define the local changes of variables. We have for \(\z = -g(-ix)\) and \(\w = g(iy)\)
\[f(\z) +\frac{2}{3}i = -x^2, \quad -\e < x < \e,\]
\[f(\w) -\frac{2}{3}i = y^2, \quad -\e < y < \e,\]
so that \eqref{eq:S11} becomes
\begin{align}\label{eq:S12} \frac{e^{-i\frac{4}{3} a^{\frac{3}{2}}}}{( 2\pi i)^2}  \int_{-\e}^{\e} dx \int_{-\e}^{\e} dy ~ \exp\left\{- a^{3/2}\left( x^2 +y^2\right)\right\}\varPhi^* (x,y),
\end{align}
where the function \(\varPhi^*(x,y)\) is given by
\begin{align}\label{varPhi*} \varPhi^*(x,y) = \frac{\exp\left\{ vg(-ix) -\t_2 (g(-ix))^2+u g(iy) +\t_1 (g(iy))^2 \right\}}{-g(-ix)-g(iy)}\frac{d}{dx}\left( -g(-ix) \right) \frac{d}{dy}g(iy).
\end{align}
We can verify without difficulty, using the fact \(\overline{g(\overline{x})} = -g(ix)\) for all \(x\in U_{\e}(0)\), that we have 
\[ \overline{\varPhi (\overline{x}, \overline{y})} = \varPhi^*(-x,-y),\]
where the function \(\varPhi(x,y)\) is defined in \eqref{varPhi}.
This is sufficient to link the coefficients of the series representation of \(\varPhi^*(x,y)\) to those of the series representation of \(\varPhi(x,y)\) by complex conjugation and negation, so that we have
\begin{align}\label{varPhi*_exp} \varPhi^* (x,y)= \sum_{k,l \geq 0} (-1)^{k+l}\overline{c_{k,l}} x^k y^l.
\end{align}
Substituting \eqref{varPhi*_exp} into \eqref{eq:S12} and integrating termwise, we obtain the following asymptotic expansion of the saddle point \((-i,i)\) to  \(J_{\t_1, \t_2}^{a}(u,v)\) 
\begin{align} &\frac{1}{( 2\pi i)^2}  \int_{\mathrm{S}(-i)} d\zeta \int_{\mathrm{T}(i)} d\w ~ \exp\left\{ a^{3/2}\left( f(\z) - f(\w)\right)\right\}\frac{\exp\lb  -v\zeta -\t_2 \zeta^2 + u \w + \t_1 \w^2\rb}{\zeta-\w}\\ 
&\approx \sum_{k,l \geq 0} \frac{(-1)^{k+l} \overline{c_{k,l}} e^{-i\frac{4}{3} a^{\frac{3}{2}}}}{a^{\frac{3}{4} (k+l+2)}} \frac{  C_k C_l }{( 2\pi i)^2} \label{eq:SS13} =\sum_{k,l \geq 0} \frac{ \overline{c_{2k,2l}} e^{-i\frac{4}{3} a^{\frac{3}{2}}}}{a^{\frac{3}{2} (k+l+1)}} \frac{ \Gamma(k+\frac{1}{2})\Gamma(l+\frac{1}{2}) }{( 2\pi i)^2}\\
&=  \sum_{\nu=1}^{\infty} \frac{1}{(2\pi i)^2 a^{\frac{3}{2} \nu}} \sum_{\substack{k,l \geq 0\\ k+l = \nu -1}} \overline{c_{2k,2l}}e^{i\frac{4}{3} a^{\frac{3}{2}}}\Gamma\left(k+\frac{1}{2}\right)\Gamma\left(l +\frac{1}{2}\right) ,\label{eq:S13}
\end{align}
as \(a\to +\infty\), uniformly in \(u,v,\t_1, \t_2\), each coming from a compact subset of \(\mathbb{R}\), where the coefficients \(C_k\) are given in \eqref{def:C}.

\subsection{Contribution from \((-i,- i)\)} For the last contribution, we look at a neighborhood of \((-i,-i )\).  We consider the integral
\begin{align} &\frac{1}{( 2\pi i)^2}  \int_{\mathrm{S}(-i)} d\zeta \int_{\mathrm{T}(-i)} d\w ~ \exp\left\{ a^{3/2}\left( f(\z) - f(\w)\right)\right\}\frac{\exp\lb  -v\zeta -\t_2 \zeta^2 + u \w + \t_1 \w^2\rb}{\zeta-\w},\qquad \label{eq:S14}
\end{align}
where \(\mathrm{S}(-i)\) and \(\mathrm{T}(-i)\) are the parts of the steepest descent paths \(\mathrm{S}\) and \(\mathrm{T}\) that intersect with a small complex neighborhood of \(-i\). 
Again, we use the function \(g\) defined in \eqref{algequ} in order to define the local changes of variables. We have for \(\z = -g(-ix)\) and \(\w = -g(-y)\)
\[f(\z) +\frac{2}{3}i = -x^2, \quad -\e < x < \e,\]
\[f(\w) +\frac{2}{3}i = y^2, \quad -\e < y < \e,\]
so that under this change of variables \eqref{eq:S14} becomes
\begin{align}\label{eq:S15} \frac{1}{( 2\pi i)^2}  \int_{-\e}^{\e} dx \int_{-\e}^{\e} dy ~ \exp\left\{- a^{3/2}\left( x^2 +y^2\right)\right\}\frac{\varphi^*(x,y)}{x+iy},
\end{align}
where the function \(\varphi^*(x,y)\) is given by
\begin{align}\label{varphi^*} \varphi^*(x,y) = &\frac{x+iy}{-g(-ix)+g(-y)}\exp\left\{ vg(-ix) -\t_2 (g(-ix))^2-u g(-y) +\t_1 (g(-y))^2 \right\}\\
&\times\frac{d}{dx}\left(-g(-ix)\right) \frac{d}{dy}\left(-g(-y)\right).
\end{align}
Using \(\overline{g(\overline{x})} = -g(ix)\) for all \(x\in U_{\e}(0)\), it follows easily that we have 
\[ \overline{\varphi (\overline{x}, \overline{y})} = -\varphi^*(-x,-y),\]
where the function \(\varphi(x,y)\) is defined in \eqref{varphi}. From this we infer that the coefficients of the series representation of \(\varphi^*(x,y)\) are linked to those of the series representation of \(\varphi(x,y)\) by conjugation and negation, so that we have
\begin{align}\label{varphi*_exp} \varphi^* (x,y)= \sum_{k,l \geq 0} (-1)^{k+l+1}\overline{b_{k,l}} x^k y^l.
\end{align}
Substituting \eqref{varphi*_exp} into \eqref{eq:S15}, and interchanging integration and summation, we obtain the following contribution of the saddle point \((-i,-i)\) to  \(J_{\t_1, \t_2}^{a}(u,v)\)
\begin{align}& \frac{1}{( 2\pi i)^2}  \int_{\mathrm{S}(-i)} d\zeta \int_{\mathrm{T}(-i)} d\w ~ \exp\left\{ a^{3/2}\left( f(\z) - f(\w)\right)\right\}\frac{\exp\lb  -v\zeta -\t_2 \zeta^2 + u \w + \t_1 \w^2\rb}{\zeta-\w}\\
&\approx \sum_{k,l \geq 0} \frac{(-1)^{k+l+1}\overline{b_{k,l}}}{a^{\frac{3}{4} (k+l+1)}}  \frac{\overline{B_{k,l}}}{( 2\pi i)^2} \label{eq:SS16}\\
&=  \sum_{k,l \geq 0, ~k+l ~\text{odd}} \frac{\overline{b_{k,l}}}{a^{\frac{3}{4} (k+l+1)}}  \frac{\overline{B_{k,l}}}{( 2\pi i)^2} =\sum_{\nu=1}^{\infty} \frac{1}{(2\pi i)^2 a^{\frac{3}{2} \nu}} \sum_{\substack{k,l \geq 0\\ k+l = 2\nu -1}} \overline{b_{k,l} B_{k,l}} ,\label{eq:S16}
\end{align}
as \(a \to +\infty\), uniformly in \(u,v,\t_1, \t_2\), each coming from a compact subset of \(\mathbb{R}\), where the coefficients \(B_{k,l}\) are defined in \eqref{def:B}.

\subsection{The complete expansion} We now are ready to collect all contributions and to arrive at a complete asymptotic expansion for the integral \(J_{\t_1, \t_2}^{a}(u,v)\) of \eqref{def:J} given by
\begin{align} \frac{1}{( 2\pi i)^2}  \int_{\mathrm{S}} d\zeta \int_{\mathrm{T}} d\w ~ \exp\left\{ a^{3/2}\left( f(\z) - f(\w)\right)\right\}\frac{\exp\lb  -v\zeta -\t_2 \zeta^2 + u \w + \t_1 \w^2\rb}{\zeta-\w}.
\end{align}
First, we fix compact subsets for \(u,v,\t_1, \t_2\), and observe that the contributions to the asymptotic behavior of \(J_{\t_1, \t_2}^{a}(u,v)\) coming from the parts of the contours away from the four saddle points are of order \(\mathcal{O}\left(e^{- \tilde{c} a^{\frac{3}{2}}}\right)\), as \(a \to +\infty\), where \(\tilde{c}\) is a positive constant, and this is valid uniformly with respect to \(u,v,\t_1, \t_2\). Adding the contributions from \eqref{eq:S5}, \eqref{eq:S10}, \eqref{eq:S13} and \eqref{eq:S16}, we find the asymptotic expansion
\begin{align} J_{\t_1, \t_2}^{a}(u,v) \approx & \sum_{\nu=1}^{\infty} \frac{2}{(2\pi i)^2 a^{\frac{3}{2} \nu}} \sum_{\substack{k,l \geq 0\\ k+l = 2\nu -1}} \Re\left\{b_{k,l} B_{k,l}\right\}\\
+&\sum_{\nu=1}^{\infty} \frac{2}{(2\pi i)^2 a^{\frac{3}{2} \nu}} \sum_{\substack{k,l \geq 0\\ k+l = \nu -1}} \Re\left\{c_{2k,2l}e^{i\frac{4}{3} a^{\frac{3}{2}}}\right\}\Gamma\left(k+\frac{1}{2}\right)\Gamma\left(l +\frac{1}{2}\right), \label{Complete_Airy}
\end{align}
as \(a \to +\infty\), uniformly in \(u,v,\t_1, \t_2\), each coming from a compact subset of \(\mathbb{R}\), for some constant \(c>0\). This leads to \eqref{Exp:AirytoS}. The statement in \eqref{trans:AirytoS} now follows by the evaluation of the first terms of this expansion. As we have \(B_{0,0}=0\), the coefficient \(b_{0,0}\) does not appear. Moreover, a computation using \eqref{varphi} gives
\begin{align} b_{1,0} = \frac{\partial \varphi}{\partial x} (x,y) \bigg\vert_{(x,y )= (0,0)}= e^{(u-v)i -(\t_1 -\t_2)} \left\{ v + i\left(2\t_2 -\frac{1}{6}\right)\right\},
\end{align}
and
\begin{align} b_{0,1} = \frac{\partial \varphi}{\partial y}(x,y) \bigg\rvert_{(x,y )= (0,0)} = e^{(u-v)i -(\t_1 -\t_2)} \left\{ -iu + 2\t_1 +\frac{1}{6}\right\}.
\end{align}
This gives us
\[ \sum_{k,l \geq 0,~ k+l = 1} \frac{2\Re\left\{  b_{k,l} B_{k,l}\right\}}{(2\pi i)^2  a^{\frac{3}{4} (k+l+1)}} = \frac{\pi}{(2\pi i)^2} \frac{ e^{-(\t_1 -\t_2)}}{a^{\frac{3}{2}}}\left\{\cos(u-v)(u+v)-2(\t_1+\t_2 )\sin(u-v)\right\},\]
and evaluating the first term corresponding to \(k=l=0\) of the second series in \eqref{Complete_Airy} gives
\[\frac{2 \Re\left\{c_{0,0} e^{i\frac{4}{3} a^{\frac{3}{2}}}\right\} \left(\Gamma(\frac{1}{2})\right)^2}{(2\pi i)^2 a^{\frac{3}{2}}} = \frac{\pi}{(2\pi i)^2} \frac{e^{-(\t_2 -\t_2)}}{a^{\frac{3}{2}}} \cos\left(\frac{4}{3} a^{\frac{3}{2}} - (u+v)\right).\]
This leads to statement \eqref{trans:AirytoS}.

\end{proof}

\begin{proof}[Proof of Theorem \ref{thm:Pearceytosine}] The proof follows the same method as the previous proof, but with some significant differences, which we will explain in the following. The heat kernel in \eqref{eq:Pearceykernel} becomes \(a^{1/3}\) times the heat kernel in \eqref{S_2-kernel} under the substitutions given in \eqref{trans:PearceytoS} \(\t_i \to 2\t_i/a^{2/3}\), \(u \to\frac{u}{a^{1/3}} +a \) and \(v \to \frac{v}{a^{1/3}} +a\). Hence, we can focus on the double contour integral in \eqref{eq:Pearceykernel}, which we again denote by \(I_{\t_1, \t_2}(u,v)\).  Again using the transformations of the variables given in \eqref{trans:PearceytoS} we have

\begin{align}&\frac{1}{a^{\frac{1}{3}}}I_{\frac{2\t_1}{a^{2/3}}, \frac{2\t_2}{a^{2/3}}}\left(\frac{u}{a^{\frac{1}{3}}}+a,\frac{v}{a^{\frac{1}{3}}}+a \right) \\ 
&=\frac{1}{a^{\frac{1}{3}}( 2\pi i)^2} \int_{i\mathbb{R}} d\zeta \int_{\Gamma^{\rm P}} d\w ~ \frac{\exp\lb -\frac{\zeta^4}4-\frac{\t_2}{a^{2/3}}\zeta^2 - \left(\frac{v}{a^{1/3}}+a\right) \zeta  +\frac{\w^4}4+  \frac{\t_1}{a^{2/3}} \w^2 +\left( \frac{u}{a^{1/3}} +a\right)\w \rb}{\zeta-\w}\\
&=\frac{1}{( 2\pi i)^2} \int_{i\mathbb{R}} d\zeta \int_{\Gamma^{\rm P}} d\w ~\frac{\exp\left(-a^{4/3}\left(\frac{\zeta^4}{4} + \zeta\right)+a^{4/3} \left(\frac{\w^4}{4} + \w\right) - v \zeta - \t_2 \zeta^2 + u\w + \t_1 \w^2\right)}{\zeta-\w},
\end{align}
where in the last step we have rescaled  \(\zeta \) and \(\w\) by  \(a^{1/3}\). For the next step, we recall that the integration contour for \(\zeta\) consists of a straight line connecting \(\infty e^{- i\pi/2 }\) and  \(\infty e^{+ i\pi/2 }\). Moreover, the integration contour for \(\w\) consists of four rays, two from the origin to $\pm\infty e^{-i\pi/4}$ and two from $\pm\infty e^{i\pi/4}$ to the origin, see Figure \ref{Pearcey_contours}. We now shift the \(\zeta\) contour horizontally to the right by \(\lambda >0\) so that the contours meet at the points \(T_1\) (in the lower right half-plane) and \(T_2\) (in the upper right half-plane). This gives us

\begin{align}&\frac{1}{a^{\frac{1}{3}}}I_{\frac{2\t_1}{a^{2/3}}, \frac{2\t_2}{a^{2/3}}}\left(\frac{u}{a^{\frac{1}{3}}}+a,\frac{v}{a^{\frac{1}{3}}}+a \right) \\ 
&=\frac{1}{( 2\pi i)^2} \int_{\lambda +  i\mathbb{R}} d\zeta \int_{\Gamma^{\rm P}} d\w ~\frac{\exp\left(-a^{4/3}\left(\frac{\zeta^4}{4} + \zeta\right)+a^{4/3} \left(\frac{\w^4}{4} + \w\right) - v \zeta - \t_2 \zeta^2 + u\w + \t_1 \w^2\right)}{\zeta-\w}\\
&\quad + \frac{1}{2\pi i} \int_{T_1}^{T_2}  ~ e^{(\t_1 - \t_2)\w^2 +(u-v)\w} d\w, \label{eq:P_1}
\end{align}
where the single integral arises from an application of the residue theorem and Cauchy's integral theorem using the same reasoning as in \eqref{intersect}. Introducing the auxilliary function \(f(\zeta) = \frac{\zeta^4}{4} + \zeta\), we find that its critical points are given by
\[\zeta_1 = e^{i\pi /3} = \frac{1}{2} + \frac{\sqrt{3}}{2}i , \quad \zeta_2 =  e^{-i\pi /3} = \frac{1}{2} - \frac{\sqrt{3}}{2}i, \quad \zeta_3 = -1,\]
with 
\[f\left(e^{i\pi /3}\right) = \frac{3}{4} e^{i\pi /3} = \frac{3}{8}  + \frac{3\sqrt{3}}{8}i, \quad f\left(e^{-i\pi /3}\right) = \frac{3}{4} e^{-i\pi /3} = \frac{3}{8}  - \frac{3\sqrt{3}}{8}i.\]
It follows that there are 9 saddle points for the phase function \(f(\zeta) - f(\w)\), given by
\[\left\{(\zeta_j, \zeta_k) ~\vert ~ j,k = 1,2,3 \right\}.\]
However, by considering the set of relevant steepest ascent/descent paths for the phase function \(f(\zeta) - f(\w)\) and deforming the contours of integration in the representation \eqref{eq:P_1} accordingly, it turns out that we only need to deal with the asymptotic contributions from the points 
\[\left\{(\zeta_j, \zeta_k) ~\vert ~ j,k = 1,2 \right\}= \left\{\left(e^{i\pi/3}, e^{i\pi/3} \right), \left(e^{i\pi/3}, e^{-i\pi/3} \right), \left(e^{-i\pi/3}, e^{i\pi/3} \right), \left(e^{-i\pi/3}, e^{-i\pi/3} \right)\right\}.\]
As we already observed in the case of the Airy-kernel, the saddle points with different entries give asymptotic contributions with a different type of behavior than the saddle points with equal entries, because the latter ones coincide with singularities of the integrand. First, we recall that the steepest ascent and descent paths for \(f(\zeta)\) passing though the critical point \(e^{-i\pi/3}\) are characterized by the equation \(\Im\{\frac{1}{4} \z^4 + \z \}= -\frac{3 \sqrt{3}}{8}\), see Figure \ref{fig:contour_Pearcey1}. It is easily checked that the steepest ascent path for \(f(\zeta)\) is given by the branch that starts at \(\infty e^{-i\pi/2}\), passes through \(e^{-i\pi/3}\), and runs off to \(+\infty\) along the positive real line. Along this path, which we will denote by \(S_{-}\), the function \(f(\zeta) - f(e^{-i\pi/3}) =f(\zeta)  - \frac{3}{4} e^{-i\pi /3}\) descents from \(+\infty\) to zero (attained at the point \(e^{-i\pi /3}\)), and returns to \(+\infty\) . The second path passing through \(e^{-i\pi/3}\), starting at \(\infty e^{-i 3\pi/4}\) and ending at \(\infty e^{-i \pi/4}\), is the steepest ascent path for \(-f(\w)\), which we denote by \(T_{-}\). 

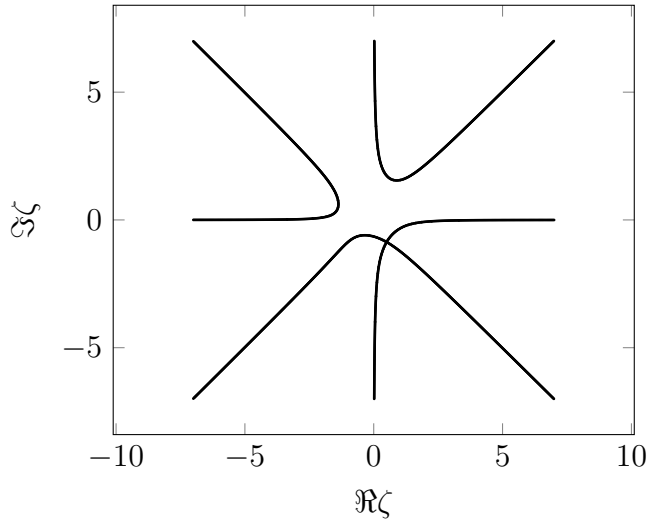
\begin{figure}[h]
    \centering
    \begin{tikzpicture}
        \begin{axis}[
            axis equal,
            xlabel={$\Re \z$},
            ylabel={$\Im \z$},
            ]
            \addplot [
                only marks,
                mark=*, 
                mark size=0.3pt,
                ] table {contour_data_Pearcey3.txt};
        \end{axis}
    \end{tikzpicture}
    \caption{Contour plot of $\Im\{\frac{1}{4} \z^4 + \z \}= -\frac{3 \sqrt{3}}{8}$ in the complex plane}
    \label{fig:contour_Pearcey1}
\end{figure}

Similarly, we can observe that the steepest descent and ascent paths of \(f(\zeta)\) through the point \(e^{i\pi/3}\) are characterized by the equation \(\Im\{\frac{1}{4} \z^4 + \z \}= \frac{3 \sqrt{3}}{8}\), see Figure \ref{fig:contour_Pearcey_2}. The steepest ascent path is given by the branch that starts at \(+\infty\), passes through \(e^{i\pi/3}\), and runs off to  \(\infty e^{i \pi/2}\), which we will denote by \(S_{+}\). The second contour through \(e^{i\pi/3}\) is the steepest ascent path for \(-f(\w)\), which we will denote by \(T_{+}\). We now use Cauchy's theorem in order to deform the contours of integration in the representation \eqref{eq:P_1} into the just found steepest ascent paths: the contour \(\lambda +  i\mathbb{R}\) we deform into the concatenation of the contours \(S_-\) and \(S_+\), and call it \(\mathrm{S}\). Analogously, we deform the part of \(\Gamma^{\rm P}\) in the lower half-plane into \(T_-\), and the part of \(\Gamma^{\rm P}\) in the upper half-plane into \(T_+\), and call the result \(\mathrm{T}\).

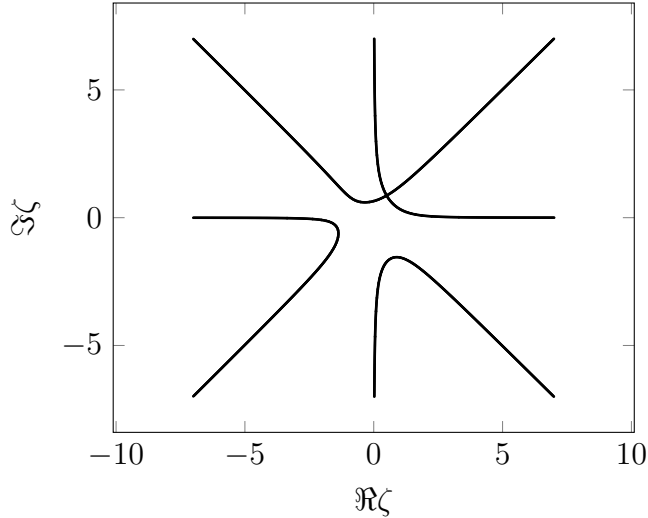
\begin{figure}[h]
    \centering
    \begin{tikzpicture}
        \begin{axis}[
            axis equal,
            xlabel={$\Re \z$},
            ylabel={$\Im \z$},
            ]
            \addplot [
                only marks,
                mark=*, 
                mark size=0.3pt,
                ] table {contour_data_Pearcey2.txt};
        \end{axis}
    \end{tikzpicture}
    \caption{Contour plot of $\Im\{\frac{1}{4} \z^4 + \z \}= \frac{3 \sqrt{3}}{8}$ in the complex plane}
    \label{fig:contour_Pearcey_2}
\end{figure}

This gives us in  \eqref{eq:P_1}
\begin{align}&\frac{1}{a^{\frac{1}{3}}}I_{\frac{2\t_1}{a^{2/3}}, \frac{2\t_2}{a^{2/3}}}\left(\frac{u}{a^{\frac{1}{3}}}+a,\frac{v}{a^{\frac{1}{3}}}+a \right) \\ 
&=\frac{1}{( 2\pi i)^2} \int_{\mathrm{S}} d\zeta \int_{\mathrm{T}} d\w ~\exp\left\{-a^{4/3}\left(f(\zeta)-f(\w)\right)\right\}\frac{\exp\lb  -v\zeta -\t_2 \zeta^2 + u \w + \t_1 \w^2\rb}{\zeta-\w}\\
&\quad + \frac{1}{2\pi i} \int_{e^{-i\pi/3}}^{e^{i\pi/3}}  ~ e^{(\t_1 - \t_2)\w^2 +(u-v)\w} d\w, \label{eq:P_2}
\end{align}
where we denote the double contour integral by \(J_{\t_1, \t_2}^{a}(u,v) \), hence
\begin{align}
	&J_{\t_1, \t_2}^{a}(u,v)\\
	&:=\frac{1}{( 2\pi i)^2} \int_{\mathrm{S}} d\zeta \int_{\mathrm{T}} d\w ~\exp\left\{-a^{4/3}\left(f(\zeta)-f(\w)\right)\right\}\frac{\exp\lb  -v\zeta -\t_2 \zeta^2 + u \w + \t_1 \w^2\rb}{\zeta-\w}.\qquad\label{def:J2}
\end{align}
In this setup, we now are able to study the asymptotic contributions from each of the four relevant saddle points in order to establish a complete expansion for \(J_{\t_1, \t_2}^{a}(u,v) \), as \(a\to\infty\).

\subsection{Contribution from \(\left(e^{i\pi/3}, e^{i\pi/3}\right)\)}
To find this contribution, we need to focus only on a small neighborhood of the point \(\left(e^{i\pi/3}, e^{i\pi/3}\right)\). To this end, we consider the integral

\begin{align}\label{eq:P_3}
\frac{1}{( 2\pi i)^2} \int\limits_{\mathrm{S}(e^{i\pi/3})} d\zeta \int\limits_{\mathrm{T}(e^{i\pi/3})} d\w ~\exp\left\{-a^{4/3}\left(f(\zeta)-f(\w)\right)\right\}\frac{\exp\lb  -v\zeta -\t_2 \zeta^2 + u \w + \t_1 \w^2\rb}{\zeta-\w},\qquad
\end{align}
where \(\mathrm{S}(e^{i\pi/3})\) and \(\mathrm{T}(e^{i\pi/3})\) are the parts of the steepest ascent paths \(\mathrm{S}\) and \(\mathrm{T}\) that intersect with a small complex neighborhood of \(e^{i\pi/3}\). As the integration contour \(\mathrm{S}(e^{i\pi/3})\) is the steepest ascent path for \(f(\z) = \frac{\z^4}{4}+ \z\) with climax at the point \(\z =e^{i\pi/3}\), and \(\z=e^{i\pi/3}\) is a simple critical point of \(f\), we can find a conformal analytic function \(g\) defined on some \(\epsilon\)-ball \(U_{\e}(0)\), mapping the interval \((-\e, \e) \subset \mathbb{R}\) bijectively onto \(\mathrm{S}(e^{i\pi/3})\) such that
\begin{equation}\label{algequ3} f(g(x))-\frac{3}{4}e^{i\pi/3} = x^2, \quad -\e < x < \e.\end{equation}
More precisely, the function \(g\) is the branch of the algebraic function \(g\) defined by \eqref{algequ3} with the series expansion
\begin{align}\label{g_exp2} g(x)= \sum_{k=0}^{\infty} a_k x^k,
\end{align}
the coefficients of which can be obtained recursively from equation \eqref{algequ3}. The first coefficients are given by \(a_0= e^{i\pi/3}, a_1 = \sqrt{\frac{2}{3}} e^{i2\pi/3}, a_2 = \frac{2}{9}, a_3 = -\frac{5 }{27} \sqrt{\frac{2}{3}} e^{i\pi/3}.\) It follows by replacing \(x\) by \(iy\) in \eqref{algequ3} that we have
\begin{equation}\label{algequ4} f(g(iy))-\frac{3}{4}e^{i\pi/3} = -y^2, \quad -\e < y < \e,\end{equation}
which shows that \(y\mapsto g(iy)\) is a conformal map from a neighborhood of the origin to a neighborhood of the point \(e^{i\pi/3}\), mapping the interval \((-\e, \e)\) onto \(\mathrm{T}(e^{i\pi/3})\). In the integral \eqref{eq:P_3}, we perform the change of variables \(\z = g(x)\) and \(\w = g(iy)\), so that we get
\begin{align}\label{eq:P_4} \frac{1}{( 2\pi i)^2}  \int_{-\e}^{\e} dx \int_{-\e}^{\e} dy ~ \exp\left\{- a^{4/3}\left( x^2 +y^2\right)\right\}\frac{\varphi(x,y)}{x-iy},
\end{align}
where the function \(\varphi(x,y)\) is given by
\begin{align}\label{varphi_P} \varphi(x,y) := \frac{x-iy}{g(x)-g(iy)}\exp\left\{ -vg(x) -\t_2 (g(x))^2+u g(iy) +\t_1 (g(iy))^2 \right\}\frac{d}{dx}g(x) \frac{d}{dy}g(iy).
\end{align}
The function \(\varphi(x,y) \) is an analytic function of two complex variables in a neighborhood of the origin \((0,0)\), so we can expand it into a power series of the form
\begin{equation}\label{varphi_exp_P}\varphi(x,y) = \sum_{k,l \geq 0}b_{k,l} x^k y^l,\end{equation}
in which the coefficients \(b_{k,l}\) depend on the variables \(u,v\) and parameters \(\t_1, \t_2\). At this point, we have to choose the size of \(\e\) small enough, so that the power series in \eqref{varphi_exp_P} converges (absolutely) in some polydisc containing \((-\e, \e)^2\). We can choose this \(\e\) independently of the variables \(u,v\) and parameters \(\t_1, \t_2\) when we restrict these quantities to come from a compact set. After substituting \eqref{varphi_exp_P} in \eqref{eq:P_4}, using the same reasoning leading to \eqref{eq:S3} and \eqref{eq:SS3}, we obtain the following asymptotic expansion for the contribution of \((e^{i\pi/3}, e^{i\pi/3})\)
\begin{align}& \frac{1}{( 2\pi i)^2}  \int_{-\e}^{\e} dx \int_{-\e}^{\e} dy ~ \exp\left\{- a^{4/3}\left( x^2 +y^2\right)\right\}\frac{\varphi(x,y)}{x-iy}\\
&\approx \sum_{k,l \geq 0} \frac{b_{k,l}}{a^{\frac{2}{3} (k+l+1)}}  \frac{1}{( 2\pi i)^2}  \int_{-\e a^{\frac{2}{3}}}^{\e a^{\frac{2}{3}}} dx \int_{-\e a^{\frac{2}{3}}}^{\e a^{\frac{2}{3}}} dy ~ \exp\left\{- \left( x^2 +y^2\right)\right\}\frac{x^k y^l}{x-iy},\label{eq:P_5}
\end{align}
as \(a \to +\infty\), uniformly in \(u,v,\t_1, \t_2\), each coming from a compact subset of \(\mathbb{R}\).
As in the Airy case we have for every \(k,l\geq 0\)
\begin{equation}\label{eq:P_6}  \int_{-\e a^{\frac{2}{3}}}^{\e a^{\frac{2}{3}}} dx \int_{-\e a^{\frac{2}{3}}}^{\e a^{\frac{2}{3}}} dy ~ \exp\left\{- \left( x^2 +y^2\right)\right\}\frac{x^k y^l}{x-iy} = B_{k,l} + \mathcal{O}\left(\left(\e a^{\frac{2}{3}}\right)^{k+l-1} e^{-\e^2 a^{\frac{4}{3}}}\right),
\end{equation}
as \(a \to +\infty\), where the constants \(B_{k,l}\) are defined in \eqref{def:B} as
\begin{equation}B_{k,l} =  \int_{-\infty}^{\infty} dx \int_{-\infty}^{\infty} dy ~ \exp\left\{- \left( x^2 +y^2\right)\right\}\frac{x^k y^l}{x-iy}, \quad k,l \geq 0.
\end{equation}
Substituting \eqref{eq:P_6} into \eqref{eq:P_5} shows now that the contribution from the saddle point \((e^{i\pi/3}, e^{i\pi/3})\) to  \(J_{\t_1, \t_2}^{a}(u,v)\) is given by
\begin{align}& \frac{1}{( 2\pi i)^2}  \int_{\mathrm{S}(e^{i\pi/3})} d\zeta \int_{\mathrm{T}(e^{i\pi/3})} d\w ~ \exp\left\{ -a^{4/3}\left( f(\z) - f(\w)\right)\right\}\frac{\exp\lb  -v\zeta -\t_2 \zeta^2 + u \w + \t_1 \w^2\rb}{\zeta-\w}\\
&\approx  \sum_{k,l \geq 0} \frac{b_{k,l}}{a^{\frac{2}{3} (k+l+1)}}  \frac{B_{k,l}}{( 2\pi i)^2} \\
&=  \sum_{k,l \geq 0, ~k+l ~\text{odd}} \frac{b_{k,l}}{a^{\frac{2}{3} (k+l+1)}}  \frac{B_{k,l}}{( 2\pi i)^2} = \sum_{\nu = 1}^{\infty} \frac{1}{ (2\pi i )^2 a^{\frac{4}{3} \nu}} \sum_{\substack{k,l \geq 0 \\ k+l = 2\nu -1  }} b_{k,l}B_{k,l}  \label{eq:P_7}
\end{align}
as \(a \to +\infty\), uniformly in \(u,v,\t_1, \t_2\), each coming from a compact subset of \(\mathbb{R}\).

\subsection{Contribution from \(\left(e^{i\pi/3}, e^{-i\pi/3}\right)\)}
To find this contribution, we are going to focus on a small neighborhood of the point \(\left(e^{i\pi/3}, e^{-i\pi/3}\right)\). To this end, we consider the integral

\begin{align}
&\frac{1}{( 2\pi i)^2} \int\limits_{\mathrm{S}(e^{i\pi/3})} d\zeta \int\limits_{\mathrm{T}(e^{-i\pi/3})} d\w ~\exp\left\{-a^{4/3}\left(f(\zeta)-f(\w)\right)\right\}\frac{\exp\lb  -v\zeta -\t_2 \zeta^2 + u \w + \t_1 \w^2\rb}{\zeta-\w}\\
=&\frac{e^{-i \frac{3\sqrt{3}}{4} a^{4/3}}}{( 2\pi i)^2} \int\limits_{\mathrm{S}(e^{i\pi/3})} d\zeta \int\limits_{\mathrm{T}(e^{-i\pi/3})} d\w ~\exp\left\{-a^{4/3}\left(f(\zeta)-\frac{3}{4} e^{i\pi/3}\right)+a^{4/3}\left( f(\w) - \frac{3}{4}e^{-i\pi/3} \right)\right\} \\
&\qquad \qquad\qquad \qquad \qquad \times  \frac{\exp\lb  -v\zeta -\t_2 \zeta^2 + u \w + \t_1 \w^2\rb}{\zeta-\w}, \label{eq:P_8}
\end{align}
where \(\mathrm{S}(e^{i\pi/3})\) and \(\mathrm{T}(e^{-i\pi/3})\) are the parts of the steepest ascent paths \(\mathrm{S}\) and \(\mathrm{T}\) that intersect with a small complex neighborhood of \(e^{i\pi/3}\) and \(e^{-i\pi/3}\), respectively. In order to construct the appropriate local transformations of the integration variables, we again make use of the function \(g\) defined in \eqref{algequ3} satisfying
\[f(g(x))-\frac{3}{4}e^{i\pi/3} = x^2 , \quad -\e <x < \e,\]
which has the series representation 
\[g(x)= \sum_{k=0}^{\infty} a_k x^k\]
and conformally maps a neighborhood of the origin onto a neighborhood of \(e^{i\pi/3}\). 
Moreover, defining 
\begin{equation}\label{def:gover}\overline{g} (x) :=\sum_{k=0}^{\infty} \overline{a_k} x^k, 
\end{equation} 
we see that \(\overline{g}\) satisfies the algebraic equation
\[f(\overline{g}(x))-\frac{3}{4}e^{-i\pi/3} = x^2, \quad -\e <x < \e,\]
from which it follows by substituting \(x= iy\), that \(y \to \overline{g}(iy)\) maps a neighborhood of the origin conformally onto a neighborhood of \(e^{-i\pi/3}\), mapping the interval \((-\e, \e)\) bijectively onto \(\mathrm{T}(e^{-i\pi/3})\), and that
\[f(\overline{g}(iy))-\frac{3}{4}e^{-i\pi/3} = -y^2, \quad -\e <x < \e.\]
 In the integral \eqref{eq:P_8}, we perform the change of variables \(\z = g(x)\) and \(\w = \overline{g}(iy)\), so that we get
\begin{align}\label{eq:P_9} \frac{e^{-i \frac{3\sqrt{3}}{4} a^{4/3}}}{( 2\pi i)^2}  \int_{-\e}^{\e} dx \int_{-\e}^{\e} dy ~ \exp\left\{- a^{4/3}\left( x^2 +y^2\right)\right\}\varPhi(x,y),
\end{align}
where the function \(\varPhi(x,y)\) is given by
\begin{align}\label{varPhi_P} \varPhi(x,y) := \frac{\exp\left\{ -vg(x) -\t_2 (g(x))^2+u \overline{g}(iy) +\t_1 (\overline{g}(iy))^2 \right\}}{g(x)-\overline{g}(iy)}\frac{d}{dx}g(x) \frac{d}{dy}\left(\overline{g}(iy)\right).
\end{align}
The function \(\varPhi(x,y) \) is an analytic function of two complex variables in a neighborhood of the origin \((0,0)\), so we can expand it into a power series of the form
\begin{equation}\label{varPhi_exp_P}\varPhi(x,y) = \sum_{k,l \geq 0}c_{k,l} x^k y^l,\end{equation}
in which the coefficients \(c_{k,l}\) depend on the variables \(u,v\) and parameters \(\t_1, \t_2\). Again, here we have to choose the size of \(\e\) small enough, so that the power series in \eqref{varPhi_exp_P} converges (absolutely) in some polydisc containing \((-\e, \e)^2\), independently of the variables \(u,v\) and parameters \(\t_1, \t_2\), each coming from a compact set.  After substituting \eqref{varPhi_exp_P} in \eqref{eq:P_9}, and by integrating termwise, we obtain for the contribution of \((e^{i\pi/3}, e^{-i\pi/3})\) the following asymptotic expansion
\begin{align}\label{eq:P_10}\sum_{k,l \geq 0}\frac{ c_{k,l}}{a^{\frac{2}{3} (k+l+2)}} \frac{e^{-i \frac{3\sqrt{3}}{4} a^{4/3}}}{( 2\pi i)^2}  \int_{-\e a^{\frac{2}{3}}}^{\e a^{\frac{2}{3}}} dx \int_{-\e a^{\frac{2}{3}}}^{\e a^{\frac{2}{3}}} dy ~ \exp\left\{-\left( x^2 +y^2\right)\right\}x^k y^l.
\end{align}
As in the Airy case, the double integrals in \eqref{eq:P_10} decouple, and we obtain for \(k,l \geq 0\)
\begin{align}\label{eq:P_11}  \int_{-\e a^{\frac{2}{3}}}^{\e a^{\frac{2}{3}}} dx \int_{-\e a^{\frac{2}{3}}}^{\e a^{\frac{2}{3}}} dy ~ \exp\left\{-\left( x^2 +y^2\right)\right\}x^k y^ll = C_k C_l +\mathcal{O}\left(e^{-\e^2 a^{\frac{4}{3}}}\right),
\end{align}
as \(a\to +\infty\), where we recall from \eqref{def:C} that
\begin{align} C_k = \int_{-\infty}^{\infty} e^{-x^2} x^k dx = \frac{1+(-1)^k}{2} \Gamma\left(\frac{k+1}{2}\right), \quad k \geq 0.
\end{align}
Substituting \eqref{eq:P_11} into \eqref{eq:P_10} we obtain for the contribution of the saddle point \((e^{i\pi/3}, e^{-i\pi/3})\) to \(J_{\t_1, \t_2}^{a}(u,v)\)
\begin{align}&\frac{1}{( 2\pi i)^2} \int\limits_{\mathrm{S}(e^{i\pi/3})} d\zeta \int\limits_{\mathrm{T}(e^{-i\pi/3})} d\w ~\exp\left\{-a^{4/3}\left(f(\zeta)-f(\w)\right)\right\}\frac{\exp\lb  -v\zeta -\t_2 \zeta^2 + u \w + \t_1 \w^2\rb}{\zeta-\w}\\
 & \approx\sum_{k,l \geq 0} \frac{c_{k,l} e^{-i\frac{3\sqrt{3}}{4} a^{\frac{4}{3}}}}{a^{\frac{2}{3} (k+l+2)}} \frac{  C_k C_l }{( 2\pi i)^2} =\sum_{k,l \geq 0} \frac{c_{2k,2l} e^{-i\frac{3\sqrt{3}}{4} a^{\frac{4}{3}}}}{a^{\frac{4}{3} (k+l+1)}} \frac{  \Gamma(k+\frac{1}{2}) \Gamma(l+\frac{1}{2}) }{( 2\pi i)^2}     \\
& = \sum_{\nu =1}^{\infty} \frac{1}{(2\pi i)^2 a^{\frac{4}{3} \nu}} \sum_{\substack{k,l\geq 0 \\ k+l = \nu -1}} c_{2k, 2l} e^{- \frac{3\sqrt{3}}{4} a^{\frac{4}{3}}}  \Gamma\left(k+\frac{1}{2}\right) \Gamma\left(l+\frac{1}{2}\right)\label{eq:P_12}
\end{align}
as \(a\to +\infty\), uniformly in \(u,v,\t_1, \t_2\), each coming from a compact subset of \(\mathbb{R}\).

\subsection{Contribution from \(\left(e^{-i\pi/3}, e^{i\pi/3}\right)\)}
We focus on a small neighborhood of the point \(\left(e^{-i\pi/3}, e^{i\pi/3}\right)\). To this end, we consider the integral 

\begin{align}
&\frac{1}{( 2\pi i)^2} \int\limits_{\mathrm{S}(e^{-i\pi/3})} d\zeta \int\limits_{\mathrm{T}(e^{i\pi/3})} d\w ~\exp\left\{-a^{4/3}\left(f(\zeta)-f(\w)\right)\right\}\frac{\exp\lb  -v\zeta -\t_2 \zeta^2 + u \w + \t_1 \w^2\rb}{\zeta-\w}\\
=&\frac{e^{i \frac{3\sqrt{3}}{4} a^{4/3}}}{( 2\pi i)^2} \int\limits_{\mathrm{S}(e^{-i\pi/3})} d\zeta \int\limits_{\mathrm{T}(e^{i\pi/3})} d\w ~\exp\left\{-a^{4/3}\left(f(\zeta)-\frac{3}{4} e^{-i\pi/3}\right)+a^{4/3}\left( f(\w) - \frac{3}{4}e^{i\pi/3} \right)\right\} \\
&\qquad \qquad\qquad \qquad \qquad \times  \frac{\exp\lb  -v\zeta -\t_2 \zeta^2 + u \w + \t_1 \w^2\rb}{\zeta-\w}, \label{eq:P_13}
\end{align}
where \(\mathrm{S}(e^{-i\pi/3})\) and \(\mathrm{T}(e^{i\pi/3})\) are the parts of the steepest ascent paths \(\mathrm{S}\) and \(\mathrm{T}\) that intersect with a small complex neighborhood of \(e^{-i\pi/3}\) and \(e^{i\pi/3}\), respectively. In order to find suitable transformations of variables, we observe that the function \(x \to \overline{g}(-x)\), using the algebraic function \(\overline{g}\) defined in \eqref{def:gover}, conformally maps a neighborhood of the origin onto a neighborhood of \(e^{-i\pi/3}\), mapping the interval \((-\e, \e)\) onto \(\mathrm{S}(e^{-i\pi/3})\) such that 
\[f(\overline{g}(-x))-\frac{3}{4}e^{-i\pi/3} = x^2, \quad -\e <x < \e.\]
Moreover, we recall from \eqref{algequ4}, that \(y\mapsto g(iy)\) is a conformal map from a neighborhood of the origin to a neighborhood of the point \(e^{i\pi/3}\), mapping the interval \((-\e, \e)\) onto \(\mathrm{T}(e^{i\pi/3})\) such that 
\begin{equation} f(g(iy))-\frac{3}{4}e^{i\pi/3} = -y^2, \quad -\e < y < \e.\end{equation}
In the integral \eqref{eq:P_13}, we perform the change of variables \(\z = \overline{g}(-x)\) and \(\w = g(iy)\), so that we get
\begin{align}\label{eq:P_14} \frac{e^{i \frac{3\sqrt{3}}{4} a^{4/3}}}{( 2\pi i)^2}  \int_{-\e}^{\e} dx \int_{-\e}^{\e} dy ~ \exp\left\{- a^{4/3}\left( x^2 +y^2\right)\right\}\varPhi^*(x,y),
\end{align}
where the function \(\varPhi^*(x,y)\) is given by
\begin{align}\label{varPhi_P*} \varPhi^*(x,y) = \frac{\exp\left\{ -v\overline{g}(-x) -\t_2 (\overline{g}(-x))^2+u g(iy) +\t_1 (g(iy))^2 \right\}}{\overline{g}(-x)-g(iy)}\frac{d}{dx}\overline{g}(-x) \frac{d}{dy}\left(g(iy)\right).
\end{align}
Recalling that \(\overline{g(\overline{x})} = \overline{g}(x)\) and using \(\overline{ \overline{g}(i\overline{y})} =g(-iy)\), we can see that 
\[ \overline{\varPhi (\overline{x}, \overline{y})} = \varPhi^*(-x,-y),\]
where the function \(\varPhi(x,y)\) is defined in \eqref{varPhi_P}.
This is sufficient to link the coefficients of the series representation of \(\varPhi^*(x,y)\) to those of the series representation of \(\varPhi(x,y)\) by complex conjugation and negation, so that we have
\begin{align}\label{varPhi*_exp_P} \varPhi^* (x,y)= \sum_{k,l \geq 0} (-1)^{k+l}\overline{c_{k,l}} x^k y^l.
\end{align}
Substituting \eqref{varPhi*_exp_P} into \eqref{eq:P_14}, and interchanging integration and summation, we obtain the following contribution of the saddle point \(\left(e^{-i\pi/3}, e^{i\pi/3}\right)\) to  \(J_{\t_1, \t_2}^{a}(u,v)\)
\begin{align}&\frac{1}{( 2\pi i)^2} \int\limits_{\mathrm{S}(e^{-i\pi/3})} d\zeta \int\limits_{\mathrm{T}(e^{i\pi/3})} d\w ~\exp\left\{-a^{4/3}\left(f(\zeta)-f(\w)\right)\right\}\frac{\exp\lb  -v\zeta -\t_2 \zeta^2 + u \w + \t_1 \w^2\rb}{\zeta-\w}\\
&\approx \sum_{k,l \geq 0} \frac{(-1)^{k+l} \overline{c_{k,l}} e^{i\frac{3\sqrt{3}}{4} a^{\frac{4}{3}}}}{a^{\frac{2}{3} (k+l+2)}} \frac{  C_k C_l }{( 2\pi i)^2} = \sum_{k,l \geq 0} \frac{ \overline{c_{2k,2l}}e^{i\frac{3\sqrt{3}}{4} a^{\frac{4}{3}}}}{a^{\frac{4}{3} (k+l+1)}} \frac{ \Gamma(k+\frac{1}{2})\Gamma(l+\frac{1}{2}) }{( 2\pi i)^2}  \\
& = \sum_{\nu =1}^{\infty} \frac{1}{(2\pi i)^2 a^{\frac{4}{3} \nu}} \sum_{\substack{k,l\geq 0 \\ k+l = \nu -1}}\overline{ c_{2k, 2l}} e^{i \frac{3\sqrt{3}}{4} a^{\frac{4}{3}}}  \Gamma\left(k+\frac{1}{2}\right) \Gamma\left(l+\frac{1}{2}\right),\label{eq:P_15}
\end{align}
as \(a\to +\infty\), uniformly in \(u,v,\t_1, \t_2\), each coming from a compact subset of \(\mathbb{R}\), where the coefficients \(C_k\) are given in \eqref{def:C}.

\subsection{Contribution from \(\left(e^{-i\pi/3}, e^{-i\pi/3}\right)\)}

For this last contribution we look at a small neighborhood of the point \(\left(e^{-i\pi/3}, e^{-i\pi/3}\right)\), considering the integral 
\begin{align}\label{eq:P_16}
\frac{1}{( 2\pi i)^2} \int\limits_{\mathrm{S}(e^{-i\pi/3})} d\zeta \int\limits_{\mathrm{T}(e^{-i\pi/3})} d\w ~\exp\left\{-a^{4/3}\left(f(\zeta)-f(\w)\right)\right\}\frac{\exp\lb  -v\zeta -\t_2 \zeta^2 + u \w + \t_1 \w^2\rb}{\zeta-\w}\qquad
\end{align}
where \(\mathrm{S}(e^{-i\pi/3})\) and \(\mathrm{T}(e^{-i\pi/3})\) are the parts of the steepest ascent paths \(\mathrm{S}\) and \(\mathrm{T}\) that intersect with a small complex neighborhood of \(e^{-i\pi/3}\). Again, we recall the function \(g\) from \eqref{algequ3}, and use the function \(\overline{g}\) from \ref{def:gover} in order to define the local changes of variables. We have for \(\z = \overline{g}(-x)\) and \(\w = \overline{g}(iy)\)

\[f(\z) -\frac{3}{4}e^{-i\pi/3} = x^2, \quad -\e < x < \e,\]
\[f(\w) -\frac{3}{4}e^{-i\pi/3} = -y^2, \quad -\e < y < \e,\]
so that under this change of variables \eqref{eq:P_16} becomes
\begin{align}\label{eq:P_17} \frac{1}{( 2\pi i)^2}  \int_{-\e}^{\e} dx \int_{-\e}^{\e} dy ~ \exp\left\{- a^{4/3}\left( x^2 +y^2\right)\right\}\frac{\varphi^*(x,y)}{x+iy},
\end{align}
where the function \(\varphi^*(x,y)\) is given by
\begin{align}\label{varphi*_P} \varphi^*(x,y) = &\frac{x+iy}{\overline{g}(-x)-\overline{g}(iy)}\exp\left\{ -v\overline{g}(-x) -\t_2 (\overline{g}(-x))^2+u\overline{g}(iy) +\t_1 (\overline{g}(iy))^2 \right\}\\
&\times\frac{d}{dx}\left(\overline{g}(-x)\right) \frac{d}{dy}\left(\overline{g}(iy)\right).
\end{align}
Using the definition on \(\overline{g}\) and the identity \(\overline{\overline{g}(i\overline{y})} = g(-iy)\) for all \(y\in U_{\e}(0)\), it follows without difficulty that we have 
\[ \overline{\varphi (\overline{x}, \overline{y})} = -\varphi^*(-x,-y),\]
where the function \(\varphi(x,y)\) is defined in \eqref{varphi_P}, from which we infer that the coefficients of the series representation of \(\varphi^*(x,y)\) are linked to those of the series representation of \(\varphi(x,y)\) by conjugation and negation, so that we have
\begin{align}\label{varphi*_exp_P} \varphi^* (x,y)= \sum_{k,l \geq 0} (-1)^{k+l+1}\overline{b_{k,l}} x^k y^l.
\end{align}
Substituting \eqref{varphi*_exp_P} into \eqref{eq:P_17}, and interchanging integration and summation, we obtain the following contribution of the saddle point \(\left(e^{-i\pi/3}, e^{-i\pi/3}\right)\) to  \(J_{\t_1, \t_2}^{a}(u,v)\)

\begin{align}& \frac{1}{( 2\pi i)^2}  \int_{\mathrm{S}(e^{-i\pi/3})} d\zeta \int_{\mathrm{T}(e^{-i\pi/3})} d\w ~ \exp\left\{ -a^{4/3}\left( f(\z) - f(\w)\right)\right\}\frac{\exp\lb  -v\zeta -\t_2 \zeta^2 + u \w + \t_1 \w^2\rb}{\zeta-\w}\\
&\approx \sum_{k,l \geq 0} \frac{(-1)^{k+l+1}\overline{b_{k,l}}}{a^{\frac{2}{3} (k+l+1)}}  \frac{\overline{B_{k,l}}}{( 2\pi i)^2} \\
&=  \sum_{k,l \geq 0, ~k+l ~\text{odd}} \frac{\overline{b_{k,l}}}{a^{\frac{2}{3} (k+l+1)}}  \frac{\overline{B_{k,l}}}{( 2\pi i)^2} = \sum_{\nu = 1}^{\infty} \frac{1}{ (2\pi i )^2 a^{\frac{4}{3} \nu}} \sum_{\substack{k,l \geq 0 \\ k+l = 2\nu -1  }}\overline{ b_{k,l}B_{k,l}}  \label{eq:P_18}
\end{align}
as \(a \to +\infty\), uniformly in \(u,v,\t_1, \t_2\), each coming from a compact subset of \(\mathbb{R}\), where the coefficients \(B_{k,l}\) are defined in \eqref{def:B}.

\subsection{The complete expansion} We now are ready to collect all contributions and to arrive at a complete asymptotic expansion for the integral \(J_{\t_1, \t_2}^{a}(u,v)\) of \eqref{def:J2}, given by
\begin{align}
	J_{\t_1, \t_2}^{a}(u,v)=\frac{1}{( 2\pi i)^2} \int_{\mathrm{S}} d\zeta \int_{\mathrm{T}} d\w ~\exp\left\{-a^{4/3}\left(f(\zeta)-f(\w)\right)\right\}\frac{\exp\lb  -v\zeta -\t_2 \zeta^2 + u \w + \t_1 \w^2\rb}{\zeta-\w}.
\end{align}
First, we fix compact subsets for \(u,v,\t_1, \t_2\), and observe that the contributions to the asymptotic behavior of \(J_{\t_1, \t_2}^{a}(u,v)\) coming from the parts of the contours away from the four saddle points are of order \(\mathcal{O}\left(e^{- \tilde{c} a^{\frac{4}{3}}}\right)\), as \(a \to +\infty\), where \(\tilde{c}\) is a positive constant, and this is valid uniformly with respect to \(u,v,\t_1, \t_2\). Adding the contributions from \eqref{eq:P_7}, \eqref{eq:P_12}, \eqref{eq:P_15} and \eqref{eq:P_18}, we find

\begin{align} J_{\t_1, \t_2}^{a}(u,v) \approx & \sum_{\nu=1}^{\infty} \frac{2}{(2\pi i)^2 a^{\frac{4}{3} \nu}} \sum_{\substack{k,l \geq 0\\ k+l = 2\nu -1}} \Re\left\{b_{k,l} B_{k,l}\right\}\\
+&\sum_{\nu=1}^{\infty} \frac{2}{(2\pi i)^2 a^{\frac{4}{3} \nu}} \sum_{\substack{k,l \geq 0\\ k+l = \nu -1}} \Re\left\{c_{2k,2l} e^{-i \frac{3\sqrt{3}}{4} a^{\frac{4}{3}}}\right\}\Gamma\left(k+\frac{1}{2}\right)\Gamma\left(l +\frac{1}{2}\right), \label{Complete_Pearcey}
\end{align}
as \(a \to +\infty\), uniformly in \(u,v,\t_1, \t_2\), each coming from a compact subset of \(\mathbb{R}\), for some constant \(c>0\). This leads to \eqref{Exp:PearceytoS}. The statement in \eqref{trans:PearceytoS} now follows by the evaluation of the first terms of this expansion. As we have \(B_{0,0}=0\), the coefficient \(b_{0,0}\) does not appear. Moreover, a computation using \eqref{varphi_P} gives
\begin{align} b_{1,0} &= \frac{\partial \varphi}{\partial x} (x,y) \bigg\vert_{(x,y )= (0,0)}\\
&=\frac{2i}{3} e^{(u-v)/2 -(\t_1 -\t_2)/2 + i \frac{\sqrt{3}}{2} \left( u-v+\t_1 -\t_2\right)} \left\{ \frac{1}{3}+\frac{v}{2}-\t_2 +\frac{\sqrt{3}i}{2}(v+2\t_2) \right\},
\end{align}
and
\begin{align} b_{0,1} &= \frac{\partial \varphi}{\partial y}(x,y) \bigg\rvert_{(x,y )= (0,0)} \\
&= \frac{2}{3} e^{(u-v)/2 -(\t_1 -\t_2)/2 + i \frac{\sqrt{3}}{2} \left( u-v+\t_1 -\t_2\right)} \left\{ -\frac{1}{3}+\frac{u}{2}-\t_1 +\frac{\sqrt{3}i}{2}(u+2\t_1) \right\}.
\end{align}
This gives us
\begin{align} &\sum_{k,l \geq 0,~ k+l = 1} \frac{2\Re\left\{  b_{k,l} B_{k,l}\right\}}{(2\pi i)^2  a^{\frac{2}{3} (k+l+1)}} = -\frac{2}{3}\frac{\pi}{(2\pi i)^2} \frac{\exp\left\{\frac{u-v}{2}-\frac{\t_1 -\t_2}{2}\right\}}{a^{\frac{4}{3}}}\\
&\times \bigg\{\left( \frac{u+v}{2} -(\t_1 + \t_2)\right)\sin\left(\frac{\sqrt{3}}{2} (u-v+\t_1 -\t_2)\right)\\
 &+ \sqrt{3}\left( \frac{u+v}{2}+ (\t_1 + \t_2)\right)\cos\left(\frac{\sqrt{3}}{2} (u-v+\t_1 -\t_2)\right)\bigg\},
\end{align}

and evaluating the first term corresponding to \(k=l=0\) of the second series in \eqref{Complete_Pearcey} gives
\begin{align}&\frac{2 \Re\left\{c_{0,0} e^{-i \frac{3\sqrt{3}}{4} a^{\frac{4}{3}}}\right\} \left(\Gamma(\frac{1}{2})\right)^2}{(2\pi i)^2 a^{\frac{4}{3}}}\\
& =\frac{4}{3\sqrt{3}} \frac{\pi}{(2\pi i)^2}  \frac{\exp\left\{\frac{u-v}{2}-\frac{\t_1 -\t_2}{2}\right\}}{a^{\frac{4}{3}}} \cos\left(\frac{3\sqrt{3}}{4} a^{\frac{4}{3}} -\frac{u+v}{2}-\frac{\t_1 +\t_2}{2}\right).
\end{align}
This leads to statement \eqref{trans:PearceytoS}.

\end{proof}

\printbibliography

\end{document}